\documentclass[]{amsart}
\usepackage[latin1]{inputenc}
\usepackage{amsmath,amsthm}
\usepackage{amscd}
\usepackage{amsfonts,enumerate}
\usepackage{amssymb}
\usepackage{mathrsfs,cite}
\usepackage{graphicx}
\usepackage{tikz}

\usepackage[all,cmtip,2cell]{xy}
\UseTwocells
\xyoption{poly}
\allowdisplaybreaks

\newtheorem{thm}{Theorem}[section]
\newtheorem{lemma}[thm]{Lemma}
\newtheorem{proposition}[thm]{Proposition}
\newtheorem{corollary}[thm]{Corollary}
\theoremstyle{definition}
\newtheorem{dfn}{Definition}[section]
\theoremstyle{remark}
\newtheorem{rmk}{Remark}[section]

\newcommand{\lel}{\mathrel{\leqslant_\ell}}
\newcommand{\ler}{\mathrel{\leqslant_{r}}}
\DeclareMathOperator{\coim}{\mathrm{coim}}

\begin{document}
	
\title{CROSS-CONNECTION STRUCTURE\\ OF LOCALLY INVERSE SEMIGROUPS}

\author{P.A. AZEEF MUHAMMED}
\address{Department of Mathematics and Natural Sciences, Prince Mohammad Bin Fahd University\\ Al Khobar 31952, Kingdom of Saudi Arabia}
\email{azeefp@gmail.com}

\author{MIKHAIL V. VOLKOV}
\address{Institute of Natural Sciences and Mathematics, Ural Federal University\\ Ekaterinburg 620000, Russia}
\email{m.v.volkov@urfu.ru}

\author{KARL AUINGER}
\address{Fakult\"{a}t f\"{u}r Mathematik, Universit\"{a}t Wien\\Oskar-Morgenstern-Platz 1\\A-1090 Wien, Austria}
\email{karl.auinger@univie.ac.at}

\maketitle
	
\begin{abstract}
Locally inverse semigroups are regular semigroups whose idempotents form pseudo-semilattices. We characterise the categories that correspond to locally inverse semigroups in the realm of Nambooripad's cross-connection theory. Further, we specialise our cross-connection description of locally inverse semigroups to inverse semigroups and completely $0$-simple semigroups, obtaining structure theorems for these classes. In particular, we show that the structure theorem for inverse semigroups can be obtained using only one category, quite analogous to the Ehresmann--Schein--Nambooripad Theorem; for completely $0$-simple semigroups, we show that cross-connections coincide with structure matrices, thus recovering the Rees Theorem by categorical tools.
\end{abstract}

\section{Introduction}\label{sec1}

Nambooripad \cite{cross}, building on an approach initiated by Hall \cite{hall} and Grillet\cite{grilcross}, introduced the notion of a normal category as the abstract categorical model of principal one-sided ideals of a regular semigroup. Each regular semigroup $S$ gives rise to two normal categories: one that models the principal left ideals of $S$ and another one that corresponds to the principal right ideals. Then Nambooripad devised a structure called  \emph{cross-connection} that captured the non-trivial interrelation between these two normal categories. Cross-connections form a category, and Nambooripad proved that this category is equivalent to the category of regular semigroups.

Thus, regular semigroups have precise categorical counterparts. A natural research program is to explore this link between semigroups and categories, by browsing through interesting families of regular semigroups and looking how their properties manifest in the language of cross-connections. Within this framework, Nambooripad's school studied cross-connections of transformation semigroups of certain concrete species, see, e.g., \cite{tlx,tx,var,raj}. However, we believe that, taking into account the abstract character of the category-based approach, it is more promising to look at specializations of Nambooripad's correspondence to important \textbf{abstract} classes of regular semigroups. To the best of our knowledge, results in this direction restrict to the following so far: the first-named author and Rajan~\cite{css} characterized cross-connections of completely simple semigroups while Rajan~\cite{inv} considered cross-connections of inverse semigroups, though a complete characterization was not achieved in the latter case.

Here we characterize cross-connections of \emph{locally inverse semigroups}. There are several reasons for us to focus on this particular class. First of all, locally inverse semigroups were introduced and deeply studied by Nambooripad himself \cite{ksslocinv1,ksslocinv2} (under the name \emph{pseudo-inverse semigroups}) so that re-visiting them appears to be appropriate for a paper dedicated to his memory. Second, the class of locally inverse semigroups is of interest and importance from both the structural viewpoint and the viewpoint of the theory of so-called e-varieties of regular semigroups.  We refer to Pastijn~\cite{pastijnlocinv} and McAlister~\cite{mcal} for elegant covering theorems that reveal the structure of locally inverse semigroups. As for the e-varietal viewpoint, the e-variety of all locally inverse semigroups is distinguished in many aspects; in particular, it is one of the two largest e-varieties which admit bi-free objects, see Yeh~\cite{Yeh}. Last but not least, we have anticipated that the categorical counterpart of such a well-behaved sort of regular semigroups should be nice in a sense, and we believe that the present paper provides some supporting evidence for these expectations.

In Section \ref{seccxnlis}, we characterise the categories of principal one-sided ideals of locally inverse semigroups as so-called unambiguous categories and describe a category equivalence between the category of cross-connected unambiguous categories and the category of locally inverse semigroups. In Section \ref{secinv}, we specialise our construction to inverse semigroups to provide a new structure theorem for inverse semigroups and also retrieve (a weaker version of) the Ehresmann--Schein--Nambooripad Theorem. This can be seen as completing the line of research initiated in~\cite{inv}. In Section~\ref{seccss}, we relate the cross-connection structure of completely $0$-simple semigroups to the Rees Theorem; it turns out that cross-connections are nothing but sandwich-matrices. This generalises the results of \cite{css} and also gives a concrete illustration for the abstract discussion of Section \ref{seccxnlis}. We conclude in Section~\ref{secconcl} by outlining directions for future work.

We assume some familiarity with basic notions from category theory (such as functors and natural transformations) and semigroup theory (such as Green's relations). For undefined notions, we refer to \cite{mac,higginscat} for category theory and \cite{clif,howie,grillet} for semigroups. Inevitably, we have followed the general line of and used several results from Nambooripad's treatise~\cite{cross} but we have made a fair effort to make the present paper understandable without studying~\cite{cross} in detail.

\section{Cross-connections of locally inverse semigroups}\label{seccxnlis}

\subsection{Inverse and locally inverse semigroups}

Two elements $a,b$ of a semigroup $S$ are said to be \emph{inverses} of each other if $aba=a$ and $bab=b$. A semigroup is called \emph{regular} if each of its elements has an inverse. \emph{Inverse semigroups} are defined as semigroups in which every element has a unique inverse; given an element $x$ in such a semigroup, its unique inverse is denoted $x^{-1}$.

An element $e$ of a semigroup $S$ is called an \emph{idempotent} if $e=e^2$, and we let $E(S)$ stand for the set of all idempotents of $S$. A \emph{band} is a semigroup in which all elements are idempotents and a \emph{semilattice} is a commutative band.

Recall a classic characterization of inverse semigroups (see, e.g., \cite[Theorem II.2.6]{grillet}).
\begin{thm}\label{thminv}
Let $S$ be a semigroup. The following are equivalent:
\begin{enumerate}
	\item $S$ is an inverse semigroup.
	\item Every $\mathscr{R}$-class of $S$ contains exactly one idempotent and every
		$\mathscr{L}$-class of $S$ contains exactly one idempotent.
	\item $S$ is regular and $E(S)$ is a semilattice.
\end{enumerate}
\end{thm}

A regular semigroup $S$ is \emph{locally inverse} if $eSe$ is an inverse semigroup for each $e\in E(S)$. We need further characterizations of locally inverse semigroups; for this purpose, we require some more definitions.

Given a semigroup $S$, we define two preorders $\lel$ and $\ler$ on the set $E(S)$ as follows: $e\lel f$ if and only if $ef=e$ and $e\ler f$ if and only if $fe=e$. The equivalences $\mathscr{L}$ and $\mathscr{R}$ on $E(S)$ induced by the preorders $\lel$ and $\ler$ respectively are nothing but the restrictions of Green's classic equivalences on $S$. The intersection $\leqslant:=\lel\cap\ler$ is a partial order on $E(S)$. We denote by $\omega(e)$ the order ideal of the set $(E(S),\leqslant)$ generated by $e$, that is, $\omega(e):=\{f\in E(S)\mid f\leqslant e\}$. Similarly, $\omega^{\ell}(e):=\{f\in E(S)\mid f\lel e\}$ and $\omega^r(e):=\{f\in E(S)\mid f\ler e\}$.
A band is said to be \emph{left} [\emph{right}] \emph{normal} if it satisfies $xyz=xzy$ [respectively, $xyz=yxz$].

The following fact comes from \cite{ksslocinv1} and \cite[Result 1]{pastijnlocinv}.
\begin{thm}\label{thmlocinv}
	Let $S$ be a regular semigroup. The following are equivalent:
	\begin{enumerate}
		\item $S$ is a locally inverse semigroup;		
        \item if $e,f,g \in E(S)$ are such that $e\mathrel{(\mathscr{L}\cup\mathscr{R})}f$ and $e,f \in \omega(g)$, then $e=f$;
		\item for each $e\in E(S)$, the set $\omega(e)$ is a semilattice;
		\item for each $e\in E(S)$, the set $\omega^{\ell}(e)$ forms a left normal band and the set $\omega^r(e)$ forms a right normal band.
	\end{enumerate}
\end{thm}

\subsection{Normal categories}
\label{ssnorcat}

For a small category $\mathcal C$, its set of objects is denoted by $v\mathcal C$ while its set of morphisms is denoted by just $\mathcal{C}$. For $c,d\in v\mathcal{C}$, the set of all morphisms from $c$ to $d$ is denoted by $\mathcal{C}(c,d)$. We compose functions and morphisms from left to right so that in expressions like $\varphi\psi$ or $\gamma\ast\delta$ etc., the left factor applies first. The identity morphism at an object $c\in v\mathcal{C}$ is denoted $1_c$.

A morphism $\varphi\colon c\to d$ is called an \emph{epimorphism} if it is left-cancellative, i.e., for all morphisms $\alpha,\beta\colon d \to e$, the equality $\varphi\alpha = \varphi\beta$ implies $\alpha=\beta$. Similarly, $\varphi\colon c\to d$ is called a \emph{monomorphism} if it is right-cancellative, i.e., for all morphisms $\alpha,\beta\colon b \to c$, the equality $\alpha\varphi=\beta\varphi$ implies $\alpha=\beta$. A morphism $\varphi\colon c\to d$ is called an \emph{isomorphism} if $\varphi\psi=1_c$ and $\psi\varphi=1_d$ for some morphism $\psi\colon d \to c$.
	
A \emph{preorder} is a category with at most one morphism from an object to another. A \emph{strict preorder} is a preorder whose only isomorphisms are the identity morphisms.

A \emph{subcategory} of a category $\mathcal{C}$ is a category $\mathcal{D}$ with $v\mathcal{D} \subseteq v\mathcal{C}$ and such that $\mathcal{D}(a,b) \subseteq \mathcal{C}(a,b)$ for all $a,b \in v\mathcal{D}$, with the same identities and composition of morphisms. $\mathcal{D}$ is a \emph{full} subcategory of $\mathcal{C}$ if $\mathcal{D}(a,b)=\mathcal{C}(a,b)$ for all $a,b \in v\mathcal{D}$.

\begin{dfn}\label{catsub}
	Let $\mathcal{C}$ be a small category and $\mathcal{P}$ a subcategory of $\mathcal{C}$ with $v\mathcal{P}=v\mathcal{C}$. The pair $(\mathcal{C},\mathcal{P})$ (often denoted by just $\mathcal{C}$) is a \emph{category with subobjects} if:
	\begin{enumerate}
		\item $\mathcal{P}$ is a strict preorder.
		\item Every morphism in $\mathcal{P}$ is a monomorphism in $\mathcal{C}$.
		\item If $\varphi,\psi\in \mathcal{P}$ and $\varphi=\alpha\psi$ for some $\alpha\in\mathcal{C}$, then $\alpha\in\mathcal{P}$.
	\end{enumerate}
\end{dfn}

In a category $(\mathcal{C},\mathcal{P})$ with subobjects, the morphisms in $ \mathcal{P}$ are called \emph{inclusions}. If $c'\to c$ is an inclusion, we write $c'\subseteq c$ and we denote this inclusion by $\iota(c',c)$. An inclusion $\iota(c',c)$ \emph{splits} if there exists $\theta\in\mathcal{C}(c,c')$ such that $\iota(c',c)\theta =1_{c'}$, and then the morphism $\theta$ is called a \emph{retraction}.

\begin{dfn}\label{normfact}
	Let $\mathcal{C}$ be a category with subobjects. A \emph{normal factorisation} of a morphism $\varphi\in\mathcal{C}$ is a decomposition of the form $\varphi=\theta\sigma\iota$, where $\theta$ is a retraction, $\sigma$ is an isomorphism, and $\iota$ is an inclusion in $\mathcal{C}$. The morphism $\theta\sigma$ is called the \emph{epimorphic component} of the morphism $\varphi$ and is denoted by $\varphi^\circ$. The codomain of $\varphi^\circ$ is called the \emph{image} of $\varphi$ and the codomain of the retraction $\theta$ is called the \emph{coimage} of $\varphi$. We denote the coimage of $\varphi$ by $\coim\varphi$.
\end{dfn}

\begin{dfn}
	Let $\mathcal{C}$ be a category with subobjects and $d\in v\mathcal{C}$. Then a function $\gamma\colon v\mathcal{C}\to \mathcal{C}$, $a\mapsto\gamma(a)\in\mathcal{C}(a,d)$ is said to be a \emph{normal cone with apex $d$} if:
	\begin{enumerate}
		\item $\iota(a,b) \gamma(b)=\gamma(a)$ whenever $a\subseteq b$;
		\item there exists at least one $c\in v\mathcal{C}$ such that $\gamma(c)\colon c\to d$ is an isomorphism.
	\end{enumerate}
\end{dfn}
For a normal cone $\gamma$, we denote by $c_\gamma$ the apex of $\gamma$ and the morphism $\gamma(c)$ is called the \emph{component} of the cone $\gamma$ at $c$.

\begin{dfn}[\!\!{\mdseries\cite[Section III.1.3]{cross}}]\label{dfnnormc}
	A category $\mathcal{C}$ is said to be \emph{normal} if:
	\begin{enumerate}
		\item [(NC 1)] $\mathcal{C}$ is a category with subobjects;
		\item [(NC 2)] every inclusion in $\mathcal{C}$ splits;
		\item [(NC 3)] every morphism in $\mathcal{C}$ admits a normal factorisation;
		\item [(NC 4)] for each $c\in v\mathcal{C}$, there exists a normal cone $\mu$ such that $\mu(c)=1_c$.
	\end{enumerate}
\end{dfn}

\subsection{Normal category of a regular semigroup}\label{ssrs}
As mentioned in the introduction, the concept of a normal category was devised by Nambooripad~\cite{cross} to capture the structure of one-sided ideals of an arbitrary regular semigroup in purely categorical terms.

Recall that for a given regular semigroup $S$, the category $\mathbb{L}(S)$ of the principal left ideals of $S$ has the set $\{ Se : e \in E(S) \}$ as the set of objects and partial right translations (i.e., right translations restricted to a principal left ideal) as morphisms. Namely, for each pair of principal left ideals $Se$ and $Sf$, the set of all morphisms from $Se$ to $Sf$ is  $ \{ \rho(e,u,f) : u\in eSf \}$, where the map $\rho(e,u,f)$
sends $x\in Se$ to $xu\in Sf$. One can easily see that $\rho(e,u,f)=\rho(g,v,h)$ if and only if $e\mathrel{\mathscr{L}}g$, $f\mathrel{\mathscr{L}}h$, and $v=gu$ \cite[Lemma III.12]{cross}. Two morphisms $\rho(e,u,f)$ and $\rho(g,v,h)$ are composable if and only $Sf=Sg$ (i.e., if and only if $f\mathrel{\mathscr{L}}g$) and then $\rho(e,u,f)\rho(g,v,h) = \rho(e,uv,h)$.

Define a subcategory $\mathcal{P}_\mathbb{L}$ of $\mathbb{L}(S)$ with $v\mathcal{P}_\mathbb{L}=v\mathbb{L}(S)$ and such that $\mathcal{P}_\mathbb{L}$ has a morphism from $Se$ to $Sf$ if and only if $Se\subseteq Sf$, in which case  there is a unique morphism, namely $\rho(e,e,f)$. The morphisms of the subcategory $\mathcal{P}_\mathbb{L}$ correspond to the inclusions of principal ideals. By definition, $\mathcal{P}_\mathbb{L}$ is a strict preorder. Clearly, every morphism $\rho(e,e,f)$ is a monomorphism. Also for morphisms $\rho(e,e,f),\rho(g,g,f)\in \mathcal{P}_\mathbb{L}$ such that $\rho(e,e,f)=\rho(h,u,g)\rho(g,g,f)$ in the category $\mathbb{L}(S)$, we have $\rho(e,e,f)=\rho(h,u,f)$, whence $u=he=h$. Therefore, the morphism $\rho(h,u,g)=\rho(h,h,g)$ belongs to $\mathcal{P}_\mathbb{L}$. Thus, all conditions of Definition~\ref{catsub} are satisfied and $(\mathbb{L}(S),\mathcal{P}_\mathbb{L})$ is a category with subobjects. We refer to morphisms in $\mathcal{P}_\mathbb{L}$ as inclusions.

Now observe that for every inclusion $\rho(e,e,f)$, we have $fe\in fSe$ so that $\rho(f,fe,e)$ is a morphism in $\mathbb{L}(S)$. Since $Se\subseteq Sf$, we have $ef=e$ whence
\[
\rho(e,e,f)\rho(f,fe,e)=\rho(e,e(fe),e)=\rho(e,(ef)e,e)=\rho(e,e,e)=1_{Se}.
\]
Thus, every inclusion in the category $\mathbb{L}(S)$ splits, and $\rho(f,fe,e)$ is a retraction.

It is shown in~\cite[Corollary III.14]{cross} that every morphism in the category $\mathbb{L}(S)$ admits a normal factorisation in the sense of Definition~\ref{normfact}.

Now, for each $a\in S$, we define a function $\rho^a\colon v\mathbb{L}(S)\to \mathbb{L}(S)$ as follows:
\begin{equation}\label{eqnprinc}
\rho^a(Se):= \rho(e,ea,f) \text{ where } f\mathrel{\mathscr{L}}a.
\end{equation}
It is easy to verify that the map $\rho^a$ is well-defined, that is, it does not depend on the particular choice of an idempotent generator of the left ideal $Se$ nor on the particular choice of an idempotent in the $\mathscr{L}$-class of $a$. Moreover, $\rho^a$ is a normal cone with apex $Sf$ in the sense of Definition~\ref{dfnicone}, see~\cite[Lemma III.15]{cross}.

In the sequel, the normal cone $\rho^a$ is called the \emph{principal cone} determined by the element $a$. In particular, observe that, for an idempotent $e\in E(S)$, we have a principal cone $\rho^e(Se)= \rho(e,e,e)=1_{Se}$.

Summarizing the above discussion, we see that the category $\mathbb{L}(S)$ satisfies all conditions (NC1)--(NC4) from the definition of a normal category (Definition~\ref{dfnnormc}).

\subsection{Normal category of a locally inverse semigroup}
\label{sslis}
Now we aim to isolate categorical properties that distinguish locally inverse semigroups within the framework of Section~\ref{ssrs}.

\begin{lemma}
\label{lemret}
If $S$ is locally inverse, every inclusion $\rho(e,e,f)\in\mathbb{L}(S)$ splits uniquely, that is, if $\varphi,\psi\in\mathbb{L}(S)$ are such that $\rho(e,e,f)\varphi=\rho(e,e,f)\psi=1_{Se}$, then $\varphi=\psi$.
\end{lemma}

\begin{proof}
Take any element $g\in fSe$ such that the morphism $\rho(f,g,e)$ satisfies $\rho(e,e,f)\rho(f,g,e)=1_{Se}$. Then $\rho(e,eg,e)=\rho(e,e,e)$ whence $eg=e$. As $g\in fSe$, we have $ge=g$ whence $g^2=(ge)g=g(eg)=ge=g$ and $g\mathrel{\mathscr{L}}e$. Besides that, we have $fg=gf=g$ since $g\in fSe$ and $ef=e$ in view of $Se\subseteq Sf$. We conclude that $g\in\omega(f)$. As $S$ is locally inverse, Theorem \ref{thmlocinv}(2) implies that each ${\mathscr{L}}$-class of $S$ has at most one idempotent in common with $\omega(f)$ so that the conditions $g\mathrel{\mathscr{L}}e$ and $g\in\omega(f)$ uniquely define $g$, and hence, the retraction $\rho(f,g,e)$.
\end{proof}

\begin{rmk}
\label{remcat}
The proof of Lemma~\ref{lemret} shows that for an arbitrary retraction $\rho(f,g,e)$, the element $g$ is an idempotent in $\omega(f)$ such that $g\mathrel{\mathscr{L}}e$. Hence $Se=Sg$ and so $\rho(f,g,e)=\rho(f,g,g)$ and $\rho(e,e,f)=\rho(g,g,f)$. We see that in $\mathbb{L}(S)$, every retraction may be represented as $\rho(f,g,g)$ and every inclusion may be represented as $\rho(g,g,f)$ for some $g\in\omega(f)$.
\end{rmk}

For an arbitrary regular semigroup $S$, morphisms of $\mathbb{L}(S)$ may admit several different normal factorisations. Our next lemma shows that locally inverse semigroups behave better in the respect.

\begin{lemma}\label{lemnormfact}
If $S$ is locally inverse, every morphism in $\mathbb{L}(S)$ has a unique normal factorisation.	
\end{lemma}

\begin{proof}
Let $\rho(e,u,f)$ with $u\in eSf$ be a morphism in the category $\mathbb{L}(S)$. Consider an arbitrary normal factorisation $\rho(e,u,f)=\theta\sigma\iota$, where $\theta$ is a retraction, $\sigma$ is an isomorphism, and $\iota$ is an inclusion in $\mathbb{L}(S)$. According to Remark~\ref{remcat}, we can write $\theta=\rho(e,g,g)$ for some $g\in\omega(e)$ and $\iota=\rho(h,h,f)$ for some $h\in\omega(f)$. Then $\sigma=\rho(g,v,h)$ for some $v\in gSh$, and we have
\[
\rho(e,u,f)=\theta\sigma\iota=\rho(e,g,g)\rho(g,v,h)\rho(h,h,f).
\]
The right hand side is equal to $\rho(e,gvh,f)=\rho(e,v,f)$ since $v\in gSh$, and we conclude that $u=v$. Thus, each normal factorisation of $\rho(e,u,f)$ can be written as
\begin{equation}
\label{eq:normalfact}
\rho(e,u,f)=\rho(e,g,g)\rho(g,u,h)\rho(h,h,f)
\end{equation}
for some $g\in\omega(e)$ and some $h\in\omega(f)$.

Since $\rho(g,u,h)$ is an isomorphism in $\mathbb{L}(S)$, there exists $u'\in hSg$ such that
\[
\rho(g,u,h)\rho(h,u',g)=1_{Sg}=\rho(g,g,g)\ \text{ and }\ \rho(h,u',g)\rho(g,u,h)=1_{Sh}=\rho(h,h,h).
\]
This implies $uu'=g$ and $u'u=h$. Since $uu'=g$, we have $g\in uS$; on the other hand, as shown above $u=v=gvh$ whence $u\in gS$. We thus conclude that $g\mathrel{\mathscr{R}}u$. Similarly, $h\mathrel{\mathscr{L}}u$. The diagram~\eqref{eq:diagram} represents the relations between the elements $e,f,g,h,u,u'$.
\begin{equation}
\label{eq:diagram}
\xymatrixcolsep{2pc}\xymatrixrowsep{1.5pc}
\xymatrix{
	e\ar@{.}[dr]^{\mathrel{\omega}}&&u\ar@{-}[dr]^{\mathrel{\mathscr{L}}}\ar@{-}[dl]_{\mathrel{\mathscr{R}}}&&f\ar@{.}[dl]_{\mathrel{\omega}}\\
	&g\ar@{-}[dr]_{\mathrel{\mathscr{L}}}&&h\ar@{-}[dl]^{\mathrel{\mathscr{R}}}\\
	&&u'
}
\end{equation}

Since $S$ is locally inverse, Theorem \ref{thmlocinv}(2) implies that each $\mathrel{\mathscr{R}}$-class of $S$ has at most one idempotent in common with $\omega(e)$ so that the conditions $g\mathrel{\mathscr{R}}u$ and $g\in\omega(e)$ uniquely define the element $g$, and consequently fix the retraction $\rho(e,g,g)$. Similarly, the conditions $h\mathrel{\mathscr{L}}u$ and $h\in\omega(f)$ uniquely define the idempotent $h$, and therefore, fix the inclusion $\rho(h,h,f)$. Moreover, as soon as $g$ and $h$ are fixed, so is the isomorphism $\rho(g,u,h)$. Altogether, the decomposition \eqref{eq:normalfact} is unique.
\end{proof}

Lemmas~\ref{lemret} and~\ref{lemnormfact} show that for $\mathbb{L}(S)$ with $S$ being locally inverse, the conditions (NC2) and (NC3) from Definition~\ref{dfnnormc} hold with uniqueness. This motivates the next definition.
\begin{dfn}
	A category $\mathcal{C}$ is said to be an \emph{unambiguous category} if:
	\begin{enumerate}
		\item [(UC 1)] $\mathcal{C}$ is a category with subobjects;
		\item [(UC 2)] every inclusion in $\mathcal{C}$ splits \emph{uniquely};
		\item [(UC 3)] every morphism in $\mathcal{C}$ admits a \emph{unique} normal factorisation;
		\item [(UC 4)] for each $c\in v\mathcal{C}$ there exists a normal cone $\mu$ such that $\mu(c)=1_c$.
	\end{enumerate}
\end{dfn}

Thus, an unambiguous category is a special version of a normal category, and we can state the following fact.

\begin{proposition}
\label{prop:unambigous}
For every locally inverse semigroup, its category of all principal left ideals is unambiguous.
\end{proposition}

\subsection{Semigroup of normal cones of an unambiguous category}\label{ssuac}

Next, we proceed to show that every unambiguous category arises as $\mathbb{L}(S)$ for some locally inverse semigroup $S$. To this end, we use Nambooripad's construction of the semigroup of normal cones of a normal category and show that it produces a locally inverse semigroup when being applied to an unambiguous category.

Given a normal category $\mathcal{C}$, we denote by $\widehat{\mathcal{C}}$ the set of all normal cones in ${\mathcal{C}}$. For $\gamma\in\widehat{\mathcal{C}}$, if $\varphi\in \mathcal{C}(c_\gamma,d)$ is an epimorphism, then as in \cite[Lemma I.1]{cross}, we can easily see that the map
\begin{equation}\label{eqnbin0}
\gamma\ast\varphi\colon c \mapsto \gamma(c)\varphi \text{ for all } c\in v\mathcal{C}	
\end{equation}
is a normal cone such that $c_{\gamma\ast\varphi}=d$. Using this observation, one can define a binary composition on $\widehat{\mathcal{C}}$ as follows: for $\gamma,\delta\in  \widehat{\mathcal{C}}$,
\begin{equation}\label{eqnbin}
\gamma \cdot\delta:=\gamma \ast(\delta(c_{\gamma}))^\circ
\end{equation}
where $(\delta(c_{\gamma}))^\circ$ is the epimorphic component of the morphism $\delta(c_{\gamma})$.

\begin{lemma}[\!\!{\mdseries\cite[Theorem I.2]{cross}}]\label{lemrs}
	Let $\mathcal{C}$ be a normal category. The set $\widehat{\mathcal{C}}$ of all its normal cones forms a regular semigroup under the binary composition \eqref{eqnbin}. A normal cone $\mu$ in $\mathcal{C}$ is an idempotent if and only if $\mu(c_\mu)=1_{c_\mu}$.
\end{lemma}

The following characterisations of the preorders $\lel,\ler$ and the order $\leqslant$ in the semigroup $\widehat{\mathcal{C}}$ can be easily extracted from the discussion in \cite[Section III.2]{cross}.
\begin{lemma}\label{lemgc}
	Let $\vartheta,\nu$ be idempotents in the semigroup $\widehat{\mathcal{C}}$. Then
	\begin{enumerate}
		\item $\nu \lel\vartheta \text{ if and only if }c_{\nu}\subseteq c_{\vartheta}.$
		\item $\nu \ler\vartheta \text{ if and only if }\nu(c_{\vartheta})\text{ is an epimorphism such that }\nu=\vartheta\ast\nu(c_{\vartheta}).$
	\end{enumerate}
\end{lemma}

\begin{lemma}\label{lemorder}
Let $\mu,\nu$ be idempotents in the semigroup $\widehat{\mathcal{C}}$. Then $\nu \leqslant \mu$ if and only if  $\nu(c_{\mu })$ is a retraction such that $\nu=\mu \ast\nu(c_{\mu }).$
\end{lemma}

\begin{proposition}
\label{prop:converse}
	Let $\mathcal{C}$ be an unambiguous category. The semigroup $\widehat{\mathcal{C}}$ of all normal cones in $\mathcal{C}$ is locally inverse.
\end{proposition}
\begin{proof}
By Lemma \ref{lemrs}, $\widehat{\mathcal{C}}$ is a regular semigroup. So, it suffices to show that $ \widehat{\mathcal{C}}$ satisfies condition (2) in Theorem \ref{thmlocinv}. Thus, let $\mu$, $\nu$, and $\vartheta$ be idempotent normal cones in $\mathcal{C}$ with apices $c$, $c_1$, and $c_2$, respectively, and $\nu,\vartheta\in\omega(\mu)$. By Lemma~\ref{lemorder}, there exist retractions $q_1:=\nu(c)\colon c\to c_1$ and $q_2:=\vartheta(c)\colon c\to c_{2}$ such that $\nu=\mu\ast q_1$ and $\vartheta=\mu\ast q_2$. We have to verify that $\nu=\vartheta$ whenever $\nu\mathrel{(\mathscr{L}\cup\mathscr{R})}\vartheta$.

First suppose that $\nu \mathrel{\mathscr{L}}\vartheta$. Then $c_1= c_2$ by Lemma~\ref{lemgc}(1). Since $\mathcal{C}$ is an unambiguous category, the inclusion $\iota({c_{1}},c)=\iota({c_{2}},c)$ splits uniquely. So, the retraction between $c$ and $c_{1}=c_{2}$ must be unique, and thus, $q_1=q_2$. Hence $\nu=\mu\ast q_1=\mu\ast q_2=\vartheta.$
	
Now, let $\nu \mathrel{\mathscr{R}}\vartheta$. By Lemma~\ref{lemgc}(2), we have $\nu=\vartheta\ast\nu(c_2)$. Equating the components at $c$, we have $\nu(c)=\vartheta(c)\nu(c_2)$, that is, $q_1=q_2\nu(c_2)$. Dually, we have $q_2=q_1\vartheta(c_1)$. Denote $u:=\vartheta(c_1)$ and $v:=\nu(c_2)$. Then, combining the equalities $q_1=q_2v$ and $q_2=q_1u$, we deduce $q_1=q_1uv$. Multiplying through on the left by the inclusion $\iota(c_1,c)$, we get $\iota(c_1,c)q_1=\iota(c_1,c)q_1uv$ whence $uv=1_{c_1}$ since $\iota(c_1,c)q_1=1_{c_1}$ by the definition of a retraction. Dually, we obtain $vu=1_{c_2}$. Thus, we see that $u$ is an isomorphism.

Observe that each of the decompositions $q_2=q_1u1_{c_2}$ and $q_2=q_21_{c_2}1_{c_2}$ constitutes a normal factorisation of the morphism $q_2$. But since $\mathcal{C}$ is an unambiguous category, every morphism has a unique normal factorisation. This implies that $q_1=q_2$ also in this case. Hence again $\nu=\mu\ast q_1=\mu\ast q_2=\vartheta.$
\end{proof}

Two unambiguous categories are said to be isomorphic if there is an inclusion preserving isomorphism between them. Our next theorem follows by restricting \cite[Theorem III.19]{cross} (which dealt with normal categories) to unambiguous categories.

\begin{thm}\label{thmls}
	Let $\mathcal{C}$ be an unambiguous category and $\widehat{\mathcal{C}}$ its associated locally inverse semigroup of normal cones. Define a functor $F\colon\mathcal{C}\to \mathbb{L}(\widehat{\mathcal{C}})$ as follows:
	\begin{align*}
	vF(c)&:= \widehat{\mathcal{C}}\mu &&\text{for }c\in v\mathcal{C},\\
	F(f)&:= \rho(\mu,\mu\ast f^\circ, \nu)&&\text{for }f\in\mathcal{C}(c,d),
	\end{align*}
	where $\mu,\nu\in E(\widehat{\mathcal{C}})$ are such that $c_\mu=c$ and $c_{\nu}=d$. Then $F$ is an isomorphism of unambiguous categories.
\end{thm}
Proposition~\ref{prop:unambigous} shows that the category  $\mathbb{L}(S)$ of a locally inverse semigroup $S$ is unambiguous. Conversely, given an abstract unambiguous category $\mathcal{C}$, Proposition~\ref{prop:converse}
shows that its semigroup $\widehat{\mathcal{C}}$ is locally inverse, and  by Theorem~\ref{thmls}, the unambiguous category $\mathbb{L}(\widehat{\mathcal{C}})$ is isomorphic to $\mathcal{C}$. That is every unambiguous category arises as the category $\mathbb{L}(S)$ of some locally inverse semigroup $S$. Thus we arrive at the following corollary that completely characterises the category of principal left ideals of a locally inverse semigroup.
\begin{corollary}\label{corls}
	A category is isomorphic to the category $\mathbb{L}(S)$ for some locally inverse semigroup $S$ if and only if it is unambiguous.
\end{corollary}

Recall from \cite{clif} that the right regular representation of a semigroup $S$ is the homomorphism $\rho\colon a\mapsto\rho_a$ of $S$ into the full transformation semigroup on the set $S$. Denote the image of $\rho$ by $S_\rho$; then $\rho\colon S\to S_\rho$ is a surjective homomorphism. The next proposition follows directly from \cite[Theorem III.16]{cross}.
\begin{proposition}\label{prosr}
Let $S$ be a locally inverse semigroup. The map $a\mapsto\rho^a$ (where $\rho^a$ is the principal cone determined by $a$) defines a homomorphism $\tilde{\rho}\colon S\to \widehat{\mathbb{L}(S)}$. Also the map $\rho_a\mapsto\rho^a$ defines an injective homomorphism $\phi\colon S_\rho\to \widehat{\mathbb{L}(S)}$ such that the next diagram commutes.
	\begin{equation*}\label{}
	\xymatrixcolsep{2pc}\xymatrixrowsep{3pc}\xymatrix
	{
		&S \ar[rd]^{\tilde{\rho}}\ar[ld]_{{\rho}} 	& \\
		S_\rho\ar[rr]_{{\phi}}&& \widehat{\mathbb{L}(S)} }
	\end{equation*}
	In particular $S$ is isomorphic to a subsemigroup of $\widehat{\mathbb{L}(S)}$ via $\tilde{\rho}$ if and only if $\rho$ is injective.
\end{proposition}

\begin{rmk}\label{rmkdual}
	Dually, we can define the unambiguous category $\mathbb{R}(S)$ of principal right ideals of a locally inverse semigroup $S$ as follows:
	\begin{align*}
		v\mathbb{R}(S)&: = \{ eS : e \in E(S) \},\\
		\mathbb{R}(S)(eS,fS)&: = \{  \lambda(e,u,f) : u\in fSe \},
	\end{align*}
where the partial left translation $\lambda(e,u,f)$ sends $x\in eS$ to $ux\in fS$. All notions defined for the category $\mathbb{L}(S)$ have their natural `duals' in the category $\mathbb{R}(S)$; in particular, $\lambda^a$ stands for the principal cone in $\mathbb{R}(S)$ determined by $a\in S$.
\end{rmk}

\subsection{Unambiguous dual and cross-connections}\label{ssecdual}

We have seen in the previous sections that given a locally inverse semigroup $S$, the categories $\mathbb{L}(S)$ and $\mathbb{R}(S)$ are unambiguous categories. In this section, we address the converse question: given two unambiguous categories $\mathcal{C}$ and $\mathcal{D}$, whether a locally inverse semigroup $S$ can be constructed such that $\mathcal{C}$ and $\mathcal{D}$ are isomorphic to $\mathbb{L}(S)$ and $\mathbb{R}(S)$, respectively. To answer this question, we first need to capture the relationship between the unambiguous categories $\mathbb{L}(S)$ and $\mathbb{R}(S)$, in a categorical framework. This is done via the notion of cross-connection which describes the relationship between the categories $\mathbb{L}(S)$ and $\mathbb{R}(S)$ using a pair of functors.

To this end, first recall that for a given normal category $\mathcal{C}$, its normal dual $\mathcal{C}^*$ is defined as a full subcategory of the functor category $[\mathcal{C},\mathbf{Set}] $ such that the objects of $\mathcal{C}^*$ are certain special set valued functors called $H$-functors.

Let $\mu$ be an idempotent cone in a normal category $\mathcal{C}$. Then we define an $H$-functor $H(\mu;-)\colon \mathcal{C}\to \mathbf{Set}$ as follows: for $c\in v\mathcal{C}$ and $g\in\mathcal{C}(c,d)$,
\begin{equation} \label{eqnH}
\begin{split}
&H({\mu};{c}):= \{\mu\ast f^\circ : f \in \mathcal{C}(c_{\mu},c)\}\\
&H({\mu};{g}) \colon H({\mu};{c})\to H({\mu};{d}) \text{ is given by }\mu\ast f^\circ \mapsto \mu\ast (fg)^\circ.
\end{split}
\end{equation}
It can be shown \cite[Lemma III.6]{cross} that the functor $H(\mu;-)$ is representable, that is, there exists an associated natural isomorphism $\eta_\mu\colon H(\mu;-) \to \mathcal{C}(c_\mu,-)$ where $\mathcal{C}(c_\mu,-)$ is the covariant hom-functor determined by the object $c_\mu$.

Now we proceed to characterise the normal dual associated with an unambiguous category. Given an unambiguous category $\mathcal{C}$, we define the \emph{unambiguous dual} $\mathcal{C}^*$ (often referred to as just \emph{dual} below) as the full subcategory of $[\mathcal{C},\mathbf{Set}] $ such that
$$v\mathcal{C}^*:=\{H(\mu;-):\mu\in E(\widehat{\mathcal{C}})\}.$$
Hence the morphisms in $\mathcal{C}^*$ are natural transformations between the $H$-functors. Then  \cite[Theorem III.25]{cross} leads us to the following theorem.
\begin{thm}\label{thmdualrs}
Let $\mathcal{C}$ be an unambiguous category with the dual $\mathcal{C}^*$. The unambiguous category $\mathbb{R}(\widehat{\mathcal{C}})$  is isomorphic to $\mathcal{C}^*$. In particular, the unambiguous dual $\mathcal{C}^*$ is also an unambiguous category.
\end{thm}

Now, we proceed to describe how the unambiguous categories $\mathbb{L}(S)$ and $\mathbb{R}(S)$ arising from the same locally inverse semigroup $S$ are interrelated. This interrelation is captured by  a pair of functors $\Gamma_S$ and $\Delta_S$. The functor $\Gamma_S\colon  \mathbb{R}(S) \to \mathbb{L}(S)^*$ is defined as follows: for each $eS\in v\mathbb{R}(S)$ and for each morphism $ \lambda(e,u,f)\in \mathbb{R}(S)$,
\begin{equation} \label{eqngs}
v\Gamma_S(eS):= H(\rho^e;-) \quad\text{ and }\quad \Gamma_S(\lambda(e,u,f)):= \eta_{\rho^e}\mathbb{L}(S)(\rho(f,u,e),-)\eta_{\rho^f}^{-1}
\end{equation}
where $\eta_{\rho^e}$ is the natural isomorphism associated with the $H$-functor $H(\rho^e;-)$. Similarly,  we define $\Delta_S\colon  \mathbb{L}(S) \to \mathbb{R}(S)^*$ as follows: for each $Se\in v\mathbb{L}(S)$ and for each morphism $\rho(e,u,f)\in \mathbb{L}(S)$,
\begin{equation}\label{eqnds}
v\Delta_S(Se):= H( \lambda^e;-) \quad\text{ and }\quad \Delta_S(\rho(e,u,f)):= \eta_{ \lambda^e}\mathbb{R}(S)( \lambda(f,u,e),-)\eta_{\lambda^f}^{-1}.
\end{equation}
As in \cite[Theorem IV.2]{cross}, we can prove that $\Gamma_S$ and $\Delta_S$ are well defined covariant functors which are inclusion preserving and fully faithful. Moreover, for each $eS\in v\mathbb{R}(S)$, the restriction
of $\Gamma_S$ to the full subcategory of $\mathbb{R}(S)$ whose objects are the principal right ideals contained in $eS$ is an isomorphism; a similar property holds for $\Delta_S$ and each $Se\in v\mathbb{L}(S)$.

The latter observation motivates the notion of a local isomorphism. An \emph{ideal} $ (  c  ) $ of an unambiguous category $\mathcal{C}$ is the full subcategory of $\mathcal{C}$ with objects
$$v (  c  ):=\{d\in v\mathcal{C}: d\subseteq c\}.$$
\begin{dfn}\label{dfnlociso}
A functor $F$ between two unambiguous categories $\mathcal{C}$ and $\mathcal{D}$ is said to be a \emph{local isomorphism} if $F$ is inclusion preserving, fully faithful and for each $c \in v\mathcal{C}$, $F_{| (  c  ) }$ is an isomorphism of the ideal $ (  c  ) $ onto $  (  F(c)  )  $.
\end{dfn}
Thus, $\Gamma_S$ and $\Delta_S$ are local isomorphisms.

To describe the relationship between the local isomorphisms $\Gamma_S$ and $\Delta_S$, we need to employ the notion of the $M$-set $MH(\mu;-)$ of an $H$-functor $H(\mu;-)$ in an unambiguous category $\mathcal{C}$. It is defined as follows:
\begin{equation}\label{eqnms}
MH(\mu;-):= \{ c\in v\mathcal{C}: \mu(c) \text{ is an isomorphism} \}.
\end{equation}
It can be seen that for objects $Se\in v\mathbb{L}(S)$ and $eS\in v\mathbb{R}(S)$,
\begin{equation}
Se \in M\Gamma_S(eS) \text{ if and only if } eS\in M\Delta_S(Se).
\end{equation}
Here the $M$-set $M\Gamma_S(eS)$ is $MH(\rho^e;-)$ according to \eqref{eqngs}; similarly, $M\Delta_S(Se)=MH( \lambda^e;-)$ by \eqref{eqnds}.

The above leads us to the definition of a cross-connection.
\begin{dfn} \label{ccxn}
	Let $\mathcal{C}$ and $\mathcal{D}$ be unambiguous categories. A \emph{cross-connection} between $\mathcal{C}$ and $\mathcal{D}$ is a quadruplet $(\mathcal{C},\mathcal{D};{\Gamma},\Delta)$ where $\Gamma\colon  \mathcal{D} \to \mathcal{C}^*$ and $\Delta\colon  \mathcal{C} \to \mathcal{D}^*$ are local isomorphisms such that for $c \in v\mathcal{C}$ and $d \in v\mathcal{D}$,
	\begin{equation}\label{eqncxnms}
	c \in M\Gamma(d) \iff d\in M\Delta(c).
	\end{equation}
\end{dfn}

Summarising the above discussion, we have proved the following theorem.
\begin{thm}\label{thmcxns}
	Let $S$ be a locally inverse semigroup with unambiguous categories $\mathbb{L}(S)$ and $\mathbb{R}(S)$. Define functors $\Gamma_S$ and $\Delta_S$ as in \eqref{eqngs} and \eqref{eqnds}. Then $\Omega S= (\mathbb{L}(S),\mathbb{R}(S);\Gamma_S,\Delta_S)$ is a cross-connection between  $\mathbb{L}(S)$ and $\mathbb{R}(S)$.
\end{thm}

\subsection{Locally inverse semigroup of a cross-connection}\label{seccxn}

We have shown how a locally inverse semigroup gives rise to a cross-connection of two unambiguous categories. Now, we describe the converse construction. We build the locally inverse semigroup arising from a cross-connection between two unambiguous categories. Recall from Section \ref{ssuac} that given an unambiguous category, we have an associated locally inverse semigroup. We will be identifying the required locally inverse semigroup associated with a cross-connection as a subdirect product of the locally inverse semigroups arising from the two unambiguous categories, i.e., as a semigroup of pairs of normal cones which `respect' the cross-connection.

First, observe that given a cross-connection $\Omega=(\mathcal{C},\mathcal{D};{\Gamma},\Delta)$, by a well-known category isomorphism \cite{mac}
$$[\mathcal{C},[\mathcal{D},\mathbf{Set}]]\cong [\mathcal{C}\times\mathcal{D},\mathbf{Set}],$$
we obtain two bifunctors $\Gamma(-,-)$ and $\Delta(-,-)$ from $\mathcal{C}\times\mathcal{D}$ to $\mathbf{Set}$.

Now, given a cross-connection $\Omega=(\mathcal{C},\mathcal{D};{\Gamma},\Delta)$, the set
\begin{equation}\label{eqneo}
E_\Omega=\{ (c,d)\in v\mathcal{C}\times v\mathcal{D} : c\in M\Gamma(d)\}
\end{equation}
is the regular biordered set associated with the cross-connection $\Omega$ \cite{cross}. We show later that the set $E_\Omega$ is in fact a pseudo-semilattice.

Recall that for each $(c,d)\in E_\Omega$, there is a uniquely defined idempotent cone $\gamma(c,d)$  in the unambiguous category $\mathcal{C}$ such that
\begin{equation}\label{eqngcd}
c_{\gamma(c,d)}=c\text{ and }H(\gamma(c,d);-)=\Gamma(d).
\end{equation}
Similarly, for each pair $(c,d)\in E_\Omega$, we have a unique idempotent cone $\delta(c,d)$ in ${\mathcal{D}}$ such that
\begin{equation}
c_{\delta(c,d)}=d\text{ and }H(\delta(c,d);-)=\Delta(c).
\end{equation}

Given a cross-connection $\Omega=(\mathcal{C},\mathcal{D};{\Gamma},\Delta)$, $(c',d),(c,d') \in E_\Omega$, $f\in \mathcal{C}(c',c)$ and $g\in \mathcal{D}(d',d)$, the morphism $g$ is called the \emph{transpose} of $f$ if the morphisms $f$ and $g$ make the following diagram commute:
\begin{equation*}\label{}
\xymatrixcolsep{3pc}\xymatrixrowsep{4pc}\xymatrix
{
	c'\ar[d]_{f}&\Delta(c') \ar[d]_{\Delta(f)}\ar[rr]^{\eta_{\delta(c',d)}}   	&& \mathcal{D}(d,-)\ar[d]^{\mathcal{D}(g,-)}&d  \\
	c& \Delta(c) \ar[rr]^{\eta_{\delta(c,d')}} && \mathcal{D}(d',-)&d'\ar[u]_{g}
}
\end{equation*}

The transpose $g\in \mathcal{D}(d',d)$ is unique for a given pair of elements in $E_\Omega$. The transpose of $f\in\mathcal{C}(c',c)$ will be denoted by $f^\dagger$ in the sequel. Then we have the following theorem which is a consequence of \cite[Theorem IV.16]{cross}, in the notation introduced above.
\begin{thm}\label{thmnatiso}
Given unambiguous categories $\mathcal{C}$ and $\mathcal{D}$ and a cross-connection $\Omega=(\mathcal{C},\mathcal{D};{\Gamma},\Delta)$ with associated bifunctors $\Gamma(-,-)$ and $\Delta(-,-)$, for each $(c,d)\in v\mathcal{C}\times v\mathcal{D}$, the map $\chi(c,d)\colon\Gamma(c,d)\to \Delta(c,d)$ given by
\begin{equation*}
\chi(c,d)\colon \gamma(c',d)\ast f^\circ \mapsto \delta(c,d')\ast (f^\dagger)^\circ
\end{equation*}
is a bijection, where  $c' \in M\Gamma(d) \text{ and }  d' \in M\Delta(c)$ and $f^\dagger\in \mathcal{D}(d',d)$ is the transpose of the morphism $f\in \mathcal{C}(c',c)$. Also the map $(c,d)\mapsto\chi(c,d)$ defines a natural isomorphism between the bifunctors $\Gamma(-,-)$ and $\Delta(-,-)$.
\end{thm}

Now, as in \cite[Section IV.5.1]{cross}, we obtain the following regular subsemigroups of the semigroups  $\widehat{\mathcal{C}}$ and $\widehat{\mathcal{D}}$ from the bifunctors $\Gamma(-,-)$ and $\Delta(-,-)$, respectively.
\begin{subequations}\label{eqnug}
\begin{align}
	\widehat{\Gamma} = & \bigcup\: \{  \Gamma(c,d) : (c,d) \in v\mathcal{C} \times v\mathcal{D} \} \\
	\widehat{\Delta} = & \bigcup\: \{  \Delta(c,d) : (c,d) \in v\mathcal{C} \times v\mathcal{D} \}
\end{align}
\end{subequations}
Since $\widehat{\Gamma}$ and $\widehat{\Delta}$ are regular subsemigroups of the locally inverse semigroups  $\widehat{\mathcal{C}}$ and $\widehat{\mathcal{D}}$, respectively,  $\widehat{\Gamma}$ and $\widehat{\Delta}$ are  locally inverse semigroups.

Then we can see that a normal cone $\gamma$ belongs to $\widehat{\Gamma}$ if and only if $\gamma=\gamma(c_1,d_1)\ast u$, where $u$ is an isomorphism in $\mathcal{C}$ and $(c_1,d_1)\in E_\Omega$. Dually, a normal cone $\delta$ belongs to $\widehat{\Delta}$ if and only if $\delta=\delta(c_1,d_1)\ast u$ where $u$ is an isomorphism in $\mathcal{D}$. 	

Now, we proceed to build the cross-connection semigroup associated with the cross-connection as a subdirect product of the locally inverse semigroups $\widehat{\Gamma}$ and $\widehat{\Delta}$. Recall that $\chi$ as defined in Theorem \ref{thmnatiso} is a natural isomorphism between the bifunctors $\Gamma(-,-)$ and $\Delta(-,-)$. This natural isomorphism gives rise to a `linking' between the locally inverse semigroups $\widehat{\Gamma}$ and $\widehat{\Delta}$.
\begin{dfn}
	Given a cross-connection $\Omega=(\Gamma,\Delta;\mathcal{C},\mathcal{D})$, a normal cone $\gamma \in \widehat{\Gamma}$ is said to be \emph{linked} to $\delta \in \widehat{\Delta}$ if there is a $(c,d) \in v\mathcal{C} \times v\mathcal{D}$ such that $\gamma \in \Gamma(c,d)$ and $ \delta = \chi(c,d)(\gamma)$; we then say that the pair $(\gamma,\delta)$ is a linked pair.
\end{dfn}

Given a cross-connection $\Omega=(\Gamma,\Delta;\mathcal{C},\mathcal{D})$ of unambiguous categories $\mathcal{C}$ and $\mathcal{D}$, define the set
\begin{equation}
\mathbb{S}\Omega=\{ (\gamma,\delta) \in \widehat{\Gamma}\times \widehat{\Delta} : (\gamma,\delta) \text{ is linked}\:\}.
\end{equation}
Define an operation on $\mathbb{S}\Omega$ as follows:
\begin{equation*}
(\gamma , \delta) \circ ( \gamma' , \delta') = (\gamma \cdot \gamma' , \delta' \cdot \delta) \  \text{  for all  }(\gamma,\delta),( \gamma' , \delta') \in \mathbb{S}\Omega.
\end{equation*} 	
Suppose $(\gamma,\delta),(\gamma',\delta') \in \mathbb{S}\Omega$, then as in the \cite[Lemma IV.30]{cross}, we can show that $\gamma \cdot \gamma'$ is linked to $ \delta' \cdot \delta$. Further by \cite[Theorem IV.32]{cross}, we see that $\mathbb{S}\Omega$ is a regular semigroup called the \emph{cross-connection semigroup}  determined by $\Omega$. Then the set of idempotents of the semigroup $\mathbb{S}\Omega$ is given by:
\begin{equation*}
E(\mathbb{S}\Omega)=\{(\gamma(c,d),\delta(c,d)) : (c,d)\in E_\Omega \}
\end{equation*}
Since $\mathbb{S}\Omega$ is a regular semigroup, the set $E(\mathbb{S}\Omega)$ is a regular biordered set. By the discussion in \cite[Section V.1.2]{cross}, we can see that the set $E(\mathbb{S}\Omega)$ is regular biorder isomorphic with the set $E_\Omega$  under the map
$$(\gamma(c,d),\delta(c,d))\mapsto (c,d).$$
More precisely, as in \cite[Section V.1.2]{cross}, we can show that the quasi orders in the set $E(\mathbb{S}\Omega)=E_\Omega$ are given by:
\begin{equation*}
(c,d) \lel(c',d') \iff c\subseteq c' \ \text{ and } \ (c,d) \ler (c',d') \iff d\subseteq d'.
\end{equation*}
so that $E_\Omega$ forms a regular biordered set with the basic products and sandwich sets as described in \cite[Section V.1.2]{cross}.

\begin{thm}\label{thmcxncon}
	Given a cross-connection $\Omega=(\Gamma,\Delta;\mathcal{C},\mathcal{D})$ of unambiguous categories $\mathcal{C}$ and $\mathcal{D}$, the cross-connection semigroup $\mathbb{S}\Omega$ is locally inverse.
\end{thm}
\begin{proof}
	Observe that the cross-connection semigroup $\mathbb{S}\Omega$ is a subdirect product of two locally inverse semigroups $\widehat{\Gamma}$ and $\widehat{\Delta}$. It is well-known that a regular subdirect product of two locally inverse semigroups is locally inverse. Hence the theorem.
\end{proof}
\begin{corollary}
	Given a cross-connection $\Omega=(\Gamma,\Delta;\mathcal{C},\mathcal{D})$ of unambiguous categories $\mathcal{C}$ and $\mathcal{D}$, the biordered set $E_\Omega$ is a pseudo-semilattice.
\end{corollary}

Further, by \cite[Theorem IV.35]{cross}, we have the following theorem.
\begin{thm}
	For a cross-connection $\Omega=(\Gamma,\Delta;\mathcal{C},\mathcal{D})$ with the cross-connection semigroup $\mathbb{S}\Omega$, the unambiguous categories $\mathbb{L}(\mathbb{S}\Omega)$ and $\mathbb{R}(\mathbb{S}\Omega)$ are isomorphic to the categories $\mathcal{C}$  and $\mathcal{D}$, respectively.
\end{thm}

It is known that the category $\mathbf{LIS}$ of locally inverse semigroups forms a full subcategory of the category $\mathbf{RS}$ of regular semigroups. Similarly, we can see that the category $\mathbf{CUC}$ of cross-connections of unambiguous categories forms a full subcategory of the category $\mathbf{Cr}$ of cross-connections of normal categories. Hence, by  \cite[Theorem V.18]{cross}, we obtain the following structure theorem for locally inverse semigroups.

\begin{thm}\label{cateq}
	The category $\mathbf{LIS}$ of locally inverse semigroups is equivalent to the category $\mathbf{CUC}$ of cross-connections of unambiguous categories.
\end{thm}

\section{The inverse case}\label{secinv}
In this section, we specialise the results of Section \ref{seccxnlis} to obtain a new structure theorem for inverse semigroups. We will see that the full machinery of cross-connections is not required for this task due to the intrinsic symmetry of inverse semigroups. It also turns out that normal cones, our previous building blocks, are too general to build inverse semigroups. We replace them with certain subspecies called inversive cones. As the reader will see, employing these cones, we obtain the structure theorem for inverse semigroups using a single category. We  characterise this category as an inversive category and prove a category equivalence between the category $\mathbf{IS}$ of inverse semigroups and the category $\mathbf{IC}$ of inversive categories; as a parallel to the Ehresmann--Schein--Nambooripad Theorem (which describes a category isomorphism between inverse semigroups and inductive groupoids). Further, we will also outline how an inductive groupoid is `sitting inside' a given inversive category; thereby describing the equivalence of these approaches. Thus, we are able to partly recover the  Ehresmann--Schein--Nambooripad Theorem from our considerations of locally inverse semigroups.

\subsection{Inverse semigroups and inversive categories}
We begin by analysing the category $\mathbb{L}(S)$, where $S$ is an inverse semigroup. Our discussion will develop as follows: registering certain properties\footnote{Some of them had been discussed in \cite{inv}, but we reprove these properties here in our notation for the sake of completeness and uniformity.} of $\mathbb{L}(S)$ is interwoven with introducing appropriate abstract notions that capture these properties. Eventually, we arrive at a collection of notions that provides a complete abstract characterisation for categories of the form $\mathbb{L}(S)$ with $S$ being an inverse semigroup.

As in Section \ref{seccxnlis}, we observe that the category $\mathbb{L}(S)$ is a category with subobjects. Since $S$ is inverse, the set $E(S)$ forms a semilattice. So, the set $v\mathbb{L}(S)$ along with the partial order $\leq$ defined by:
$$Se\leq Sf \iff Se\subseteq Sf$$
also forms a semilattice. Abstracting this, we have the following definition.
\begin{dfn}\label{dfncatsem}
	 A category with subobjects $(\mathcal{C},\mathcal{P})$ is called a \emph{semilattice ordered category} (abbreviated as \emph{so-category} in the sequel) if $v\mathcal{P}$ forms a semilattice with respect to the relation $\leq$ defined as:
	$$	p\leq q \iff \text{there is an inclusion from }p \text{ to }q$$
	for any $p,q\in v\mathcal{P}$.
\end{dfn}

Given any two objects $c,c'$ in an so-category $\mathcal{C}$, there exists a unique object $d$ in $\mathcal{C}$ such that 1) there are inclusions from $d$ to $c$ and $c'$, 2) for every object $d'$ such that there are inclusions from $d'$ to $c$ and $c'$, there is an inclusion from $d'$ to $d$. Hence this `maximal' object $d$ acts as the \emph{meet} of the objects $c$ and $c'$. We denote the unique object $d$ in $v\mathcal{C}$ by $c\wedge c'$.

When $S$ is an inverse semigroup, we can easily see that any inclusion in $\mathbb{L}(S)$ will be of the form  $\rho(e,e,f)$ and any retraction will be of the form $\rho(e,f,f)$ where $e,f\in E(S)$. Any isomorphism  in $\mathbb{L}(S)$ will be of the form $\rho(aa^{-1}, a, a^{-1}a)$ for some $a\in S$; equivalently of the form $\rho(e,u,f)$ for some $e,f\in E(S)$ and some $u\in S$ such that  $e\mathrel{\mathscr{R}}u\mathrel{\mathscr{L}}f$.
For an arbitrary morphism $\rho(e,u,f)\in \mathbb{L}(S)$, we can easily see that its unique normal factorisation is given by:
$$\rho(e,u,f) = \rho(e,g,g) \rho(g,u,h) \rho(h,h,f)$$
where $g=uu^{-1}$ and $h=u^{-1}u$. Also, for an inclusion $\rho(e,e,f)$, its unique retraction is given by $\rho(f,e,e)$.

Observe that in an so-category $\mathcal{C}$ where every inclusion splits uniquely, we have two morphisms uniquely associated with any pair $c,c'\in v\mathcal{C}$ such that $c\leq c'$, namely, the inclusion $\iota(c,c')$ and its retraction $q(c',c)$. The subcategory generated by the retractions and inclusions in $\mathcal{C}$ is called the \emph{core} of $\mathcal{C}$ and  denoted by $\langle \mathcal{C} \rangle$.

When $S$ is an inverse semigroup, any morphism $\rho$ in the core category $\langle\mathbb{L}(S)\rangle$  can be written as $\rho= \rho(e_1,e_1,e_2)\:\rho(e_2,e_3,e_3)\cdots\rho(e_{2n-1},e_{2n-1},e_{2n})\:\rho(e_{2n},e_{2n},e_{2n+1})$ where each `odd' factor $\rho(e_{2i-1},e_{2i-1},e_{2i})$ is an inclusion and each `even' factor $\rho(e_{2i},e_{2i},e_{2i+1})$ is a retraction. (To ensure that the factorisation of $\rho$ starts with an inclusion and ends with a retraction, we can, if necessary, prepend or append a morphism of the form $1_e=\rho(e,e,e)$ which serves both as an inclusion  and a retraction.) Then $\rho=\rho(e_1,e_1e_3\cdots e_{2n+1}, e_{2n+1})$. Since $E(S)$ is a semilattice, we see that the element $e_1e_3\cdots e_{2n+1}$ is an idempotent, say $f$, and so the unique normal factorisation of $\rho$ will be given by
$$\rho= \rho(e_1,f,f)\rho(f,f,e_n)$$
where $\rho(e_1,f,f)$ is a retraction and $\rho(f,f,e_n)$ is an inclusion. This leads to the following definition of a special kind of normal factorisation.

\begin{dfn}\label{invfact}
Let $\mathcal{C}$ be an so-category in which inclusions split uniquely. A morphism $f=\iota(c_1,c_2)q(c_2,c_3)\cdots \iota(c_{2n-1},c_{2n})q(c_{2n},c_{2n+1})$ in the core $\langle \mathcal{C} \rangle$ is said to have an \emph{inversive factorisation} if $f=q(c_1,d)\iota(d,c_{2n+1})$ where $q(c_1,d)$ is the retraction from $c_1$ to $d:=\bigwedge\limits_{i=0}^n c_{2i+1}$ and $\iota(d,c_{2n+1})$ is the inclusion from $d$ to $c_{2n+1}$.
\end{dfn}

\begin{rmk}\label{rmkinvfact2}
Let $\mathcal{C}$ be an so-category with inclusions splitting uniquely where every morphism in the core $\langle \mathcal{C} \rangle$ has an {inversive factorisation}. Then if $f\colon c \to d$ is an isomorphism in $\langle \mathcal{C} \rangle$, then $c=d$ and $f=1_c$, that is, the only isomorphisms in $\langle \mathcal{C} \rangle$ are the identity morphisms. In other words, the core $\langle \mathcal{C} \rangle$ is the subcategory of $\mathcal{C}$ consisting of all the morphisms arising from the underlying semilattice.
\end{rmk}	

In an so-category $\mathcal{C}$, a \emph{cone} $\gamma$ with apex $d\in v\mathcal{C}$ is a map $\gamma\colon v\mathcal{C}\to\mathcal{C}$ such that $\gamma(a)\in\mathcal{C}(a,d)$ for each $a\in v\mathcal{C}$ and $\iota(a,b)\gamma(b)=\gamma(a)$ whenever $a\leq b$. For such a cone $\gamma$, define the set
$$m_\gamma:=\{ c: \gamma(c) \text{ is an isomorphism}\}.$$
We  refer to the set $m_\gamma$ as the $M$-set of the cone $\gamma$.

As mentioned earlier, normal cones prove to be too general for the inverse case. So, we proceed to analyse the principal cone $\rho^a$ in an inverse semigroup $S$ to characterise our apex building blocks. Recall from equation \eqref{eqnprinc} that the principal cone $\rho^a\colon v\mathbb{L}(S)\to \mathbb{L}(S)$ with apex $c_\gamma=Sa=Sf$ is defined for each $Se\in v\mathbb{L}(S)$ as $\rho^a(Se) = \rho(e,ea,f)$, where $f\in E(L_a)$. Since $S$ is inverse, there is a unique idempotent in the $\mathscr{L}$-class of the element $a$ so that $f=a^{-1}a$. Now, recall from \cite[Lemma III.15]{cross} that the set
$$m_{\rho^a}:=\{ Se: \rho^a(Se) \text{ is an isomorphism}\} =\{ Se : e\in E(R_a)\} = \{Saa^{-1}\}.$$
That is, for a principal cone $\rho^a$ in the category $\mathbb{L}(S)$ where $S$ is an inverse semigroup, there is a unique isomorphism component and so the $M$-set $m_{\rho^a}$ is a singleton. In what follows, we use the same notation for this singleton set and its unique element.

Let $\rho^a$ be a principal cone in the category $\mathbb{L}(S)$ and let $Sf:=Sa^{-1}a$ be its apex and $Sg:=Saa^{-1}$ its $M$-set. Then for an arbitrary object $Se\in v\mathbb{L}(S)$, we have
\begin{align*}
\rho^a(Se)&=\rho(e,ea,f)\\
&= \rho(e,eg,eg)\rho(eg,ea,h)\rho(h,h,f) &&\text{using normal factorisation and }h=a^{-1}ea\\
&= \rho(e,eg,eg)\rho(eg,ea,f)&&\text{since }eah=ea(a^{-1}ea)=ea(a^{-1}e) ea = ea\\
&= \rho(e,eg,eg)\rho(eg,ega,f)&&\text{since }ea=e(aa^{-1}a)=e(aa^{-1})a = ega\\
&= q(Se,Seg)\rho^a(Seg)&& \\
&= q(Se,Seg)\:\iota(Seg,Sg)\:\rho^a(Sg) &&\text{since }\rho^a \text{ is a cone and }Seg\subseteq Sg.
\end{align*}
That is, the component of the principal cone $\rho^a$ at any object $Se$ is composed of:
\begin{enumerate}[(i)]
	\item the retraction from $Se$ to the object $Seg=Se\wedge m_{\rho^a}$,
	\item the inclusion from $Se\wedge m_{\rho^a}$ to the object $m_{\rho^a}$,
	\item the component of the principal cone at $m_{\rho^a}$.
\end{enumerate}
Hence, $\coim\rho^a(Se)=Se\wedge m_{\rho^a}$. The above discussion inspires us to isolate certain special normal cones in an so-category, which we call inversive cones.

\begin{dfn}\label{dfnicone}
	A cone $\gamma$ in an so-category $\mathcal{C}$ is said to be an \emph{inversive cone} if:
	\begin{enumerate}
		\item the $M$-set $m_{\gamma}$ is  a singleton,
		\item for each $c\in v\mathcal{C}$, $\coim\gamma(c)=c\wedge m_{\gamma}$.
	\end{enumerate}
\end{dfn}

\begin{rmk}\label{rmkinvcone}
	Observe that an inversive cone $\gamma$ in an so-category $\mathcal{C}$ gets completely determined by the component of $\gamma$ at the object $m_\gamma\in v\mathcal{C}$. Indeed, for any $c\in v\mathcal{C}$, since $c\wedge m_\gamma \le m_\gamma$, we have $\gamma(c\wedge m_\gamma)=\iota(c\wedge m_\gamma,m_\gamma)\gamma(m_\gamma)$. Hence $\gamma(c)=q(c,c\wedge m_\gamma)\iota(c\wedge m_\gamma,m_\gamma)\gamma(m_\gamma)$. This leads to the following lemma.
\end{rmk}

\begin{lemma}\label{lempic}
	If $S$ is an inverse semigroup, for every $Sf\in \mathbb{L}(S)$, there exists a unique idempotent inversive cone with apex $Sf$, namely $\rho^f$.
\end{lemma}

\begin{proof}
	Clearly the cone $\rho^f$ is an idempotent inversive cone with apex $Sf$. Let $\mu$ be an arbitrary idempotent inversive cone with apex $Sf$. Since $\mu$ is an idempotent cone, by Lemma \ref{lemrs}, we have $\mu(Sf)=1_{Sf}$. But since $\mu$ is an inversive cone, the $M$-set $m_{\mu}$ is  a singleton and this implies that $m_{\mu}=\{Sf\}$. Also, for an arbitrary $Se\in v\mathbb{L}(S)$, we have $\coim\mu(Se)=Se\wedge m_{\mu}$. So,
	\begin{align*}
	\mu(Se)&= q(Se,Se\wedge m_{\mu}) \mu(Se\wedge m_{\mu})&&\\
	&= q(Se,Se\wedge m_{\mu})\iota(Se\wedge m_{\mu},m_{\mu})\mu(m_{\mu})&&\text{since }Se\wedge m_{\mu}\le m_{\mu}\\
	&=\rho(e,ef,ef)\rho(ef,ef,f)\rho(f,f,f)&&\text{since }m_{\mu}=\{Sf\}\\
	&=\rho(e,ef,f)&&\\
	&=\rho^f(Se).&&
	\end{align*}
	So $\mu=\rho^f$ and hence the lemma.
\end{proof}

Now, we collect all properties of the category of the principal left ideals of an inverse semigroup registered so far into a suitable categorical abstraction.
\begin{dfn}\label{dfnicat}
	A category $\mathcal{C}$ is said to be an \emph{inversive category} if:
	\begin{enumerate}
		\item [(IC 1)] $\mathcal{C}$ is an so-category;
		\item [(IC 2)] every inclusion in $\mathcal{C}$ splits \emph{uniquely};
		\item [(IC 3)] every morphism in $\mathcal{C}$ admits a \emph{unique} normal factorisation;
		\item [(IC 4)] every morphism in the core $\langle \mathcal{C} \rangle$ has an \emph{inversive} factorisation;
		\item [(IC 5)] for each $c\in v\mathcal{C}$, there is a \emph{unique inversive} idempotent cone with apex $c$.
	\end{enumerate}
\end{dfn}
In the sequel, given an object $c$ in an inversive category $\mathcal{C}$, the {unique inversive} idempotent cone with apex $c$ is denoted by $\mu_c$.

An inclusion preserving functor $F$ between two inversive categories $\mathcal{C}_1$ and $\mathcal{C}_2$ is called an \emph{inversive functor} if for any two objects $c,c'\in v\mathcal{C}_1$,
$$vF(c\wedge c')= vF(c)\wedge vF(c').$$
It is easy to see that inversive categories with inversive functors as morphisms form a locally small category $\mathbf{IC}$.

We have already seen that given an inverse semigroup $S$, the category $\mathbb{L}(S)$ is inversive. Further, given a homomorphism $\phi$ between two inverse semigroups $S_1$ and $S_2$, we can define a functor $\Phi\colon \mathbb{L}(S_1) \to \mathbb{L}(S_2)$ as follows. For idempotents $e,f\in S_1$ and $u\in eS_1f$,
\begin{equation}\label{eqnPhi}
v\Phi\colon S_1e\mapsto S_2 e\phi;\quad \Phi \colon \rho(e,u,f) \mapsto \rho(e\phi,u\phi,f\phi).
\end{equation}
It is easy to verify that $\Phi$ is an inversive functor. We have the following proposition.
\begin{proposition}\label{proC}
The assignment:
	$$S\mapsto \mathbb{L}(S); \quad \phi \mapsto \Phi$$
	constitutes a functor, say $\mathtt{C}$, from the category $\mathbf{IS}$ of inverse semigroups to the category $\mathbf{IC}$ of inversive categories.
\end{proposition}

\subsection{Inverse semigroup from an inversive category}

Having functorially associated an inversive category with a given inverse semigroup in the previous section, we proceed to show that every inversive category arises as the category $\mathbb{L}(S)$ of a suitable inverse semigroup $S$. Naturally, we search for our required inverse semigroup in the set of all inversive cones arising from an inversive category. First we need to prove some preliminary lemmas.

\begin{lemma}\label{leminvcone}
If an so-category $\mathcal{C}$ is unambiguous, then a cone $\gamma$ in $\mathcal{C}$ is inversive if and only if it can be represented as $\gamma=\mu \ast u$ where $\mu$ is an idempotent inversive cone and $u$ is an isomorphism in $\mathcal{C}$.
\end{lemma}
\begin{proof}
Let $\gamma$ be an inversive cone. Then $\gamma(m_\gamma)$ is an isomorphism. Let $u=\gamma(m_\gamma)$ and $\mu= \gamma\ast u^{-1}$. By definition $\gamma=\gamma\ast (u^{-1}u)= 	(\gamma\ast u^{-1})\ast u=\mu\ast u$. Also since $c_\mu=m_\gamma$, we have $\mu(c_{\mu})=\gamma\ast u^{-1}(m_\gamma)=\gamma(m_\gamma)\:u^{-1} =uu^{-1}=1_{c_{\mu}}$. Hence $\mu$ is an idempotent cone.
	
Now it remains to show that $\mu$ is inversive. Suppose $\mu$ does not satisfy condition (1) of Definition \ref{dfnicone}. Then there exists an object $d\in v\mathcal{C}$ such that $d\ne m_\gamma$ and $\mu(d)$ is an isomorphism. Then $\mu(d)=\gamma(d)\ast u^{-1}$ is an isomorphism. That is $\gamma(d)$ is an isomorphism and $d\ne m_\gamma$. This is a contradiction to the fact that $\gamma$ is inversive and hence $\mu $ satisfies Definition \ref{dfnicone}(1).
	
Also, for an arbitrary $c\in v\mathcal{C}$, observe that
\begin{align*}
\mu(c)&=\gamma(c)\: u^{-1}&&\\
&=q(c,c\wedge m_\gamma)\iota(c\wedge m_\gamma,m_\gamma)\gamma(m_\gamma)(\gamma(m_\gamma))^{-1}&&\text{ by Remark }\ref{rmkinvcone}\\
&=q(c,c\wedge m_\gamma)\iota(c\wedge m_\gamma,m_\gamma).&&
\end{align*}
Hence using the unique normal factorisation property, we have $\coim\mu(c)=c\wedge m_\gamma=c\wedge c_\mu=c\wedge m_\mu$. Thus, $\mu$ is an inversive cone.
	
Conversely, suppose $\gamma=\mu\ast u$ where $\mu$ is an idempotent inversive cone and $u$ is an isomorphism in $\mathcal{C}$. Then arguing similarly as above, we can see $\gamma$ is a cone satisfying Definition \ref{dfnicone}(1) with $m_\gamma=c_\mu=m_\mu$ and $\gamma(m_\gamma)=u$. Also, for an arbitrary $c\in v\mathcal{C}$, observe that
	\begin{align*}
	\gamma(c)&=\mu(c)\: u&&\\
	&=q(c,c\wedge m_\mu)\iota(c\wedge m_\mu,m_\mu)\: u&&\text{by Remark }\ref{rmkinvcone}\\
	&=q(c,c\wedge m_\gamma)\iota(c\wedge m_\gamma,m_\gamma) \gamma(m_\mu)&&\text{since }m_\gamma=m_\mu.
	\end{align*}
Since $\iota(c\wedge m_\gamma,m_\gamma) \gamma(m_\mu)$ is a monomorphism, using the unique normal factorisation property, we have $\coim\gamma(c)=c\wedge m_\gamma$. Hence $\gamma$ is an inversive cone.
\end{proof}

\begin{rmk}\label{rmkuninvcone}
In an inversive category, since there is an associated {inversive} idempotent cone $\mu_c$ for each object $c$, the above lemma can be strengthened as follows: a cone $\gamma$ is inversive if and only if $\gamma$ can be uniquely represented as $\gamma=\mu_{m_\gamma} \ast \gamma(m_\gamma)$.	
\end{rmk}

\begin{lemma}\label{lemidinvcone}
	Given an idempotent inversive cone $\mu$ and a retraction $e\colon c_\mu \to d$, the cone $\nu=\mu\ast e$ is an idempotent inversive cone with apex $d$.
\end{lemma}
\begin{proof}
	Since $\mu$ is an idempotent inversive cone $\mu(c)=q(c,c\wedge c_\mu)\iota(c\wedge c_\mu, c_\mu)$ and also $\mu(c_\mu)=1_{c_\mu}$. Then
	\begin{align*}
	\nu(d)&=\mu\ast e (d)= \mu(d)\: e\\
	&=q(d,d\wedge c_\mu)\:\iota(d\wedge c_\mu, c_\mu)\: q(c_\mu,d)	&&\text{since }e= q(c_\mu,d)\\
	&=q(d,d)\:\iota(d,c_\mu)\:q(c_\mu,d) &&\text{since }d\wedge c_\mu=d\\
	&=1_d  &&\text{since }\iota(d,c_\mu)\:q(c_\mu,d)=1_d.
	\end{align*}
	Hence $\nu$ is an idempotent cone with apex $d$. Now if there exists $d' \in m_\nu$, then $\nu(d')$ is an isomorphism. Then
	$$\nu(d')=\mu(d')\: q= q(d',d'\wedge c_\mu)\:\iota(d'\wedge c_\mu, c_\mu)\:q(c_\mu,d)).$$
	Now since $\mathcal{C}$ is inversive and $\nu(d')\in \langle \mathcal{C} \rangle$, by (IC 4) and Remark \ref{rmkinvfact2}, we have $d'=d$ and $\nu(d')=1_{d}$. Hence $\nu$ satisfies Definition \ref{dfnicone}(1).
	
	Also, $\nu(c)=q(c,c\wedge c_\mu)\iota(c\wedge c_\mu, c_\mu)q(c_\mu,c_\nu)$ and so $\nu(c)\in \langle\mathcal{C} \rangle$. By (IC 4), the morphism $\nu(c)$ has a unique inversive factorisation. Since $c\wedge c_\mu \wedge c_\nu = c\wedge (c_\mu \wedge c_\nu)= c \wedge c_\nu$, we have
	$$\nu(c)=q(c,c\wedge c_\nu)\iota(c\wedge c_\nu,c_\nu).$$
	Hence $\coim\nu(c)=c\wedge c_\nu=c\wedge m_\nu$ and the cone $\nu$ satisfies Definition \ref{dfnicone}(2).
\end{proof}

Now, given an inversive category $\mathcal{C}$, we build an inverse semigroup from its inversive cones. Let $\gamma,\delta$ be inversive cones in $\mathcal{C}$. As in \eqref{eqnbin}, we define
\begin{equation*}
\gamma \cdot\delta:=\gamma \ast(\delta(c_{\gamma}))^\circ
\end{equation*}
where $(\delta(c_{\gamma}))^\circ$ is the epimorphic component of the morphism $\delta(c_{\gamma})$. Then clearly $\gamma\cdot\delta$ is a cone. We need to verify that it is an inversive cone.
\begin{proposition}
	$\gamma\cdot\delta$ as defined above is an inversive cone.	
\end{proposition}
\begin{proof}
	By Lemma \ref{leminvcone}, it suffices to show that $\gamma\cdot\delta$ can be represented as $\mu\ast u$ where $\mu$ is an idempotent inversive cone and $u$ is an isomorphism in $\mathcal{C}$. Since $\gamma$ is an inversive cone, $\nu:=\gamma\ast (\gamma(m_\gamma))^{-1}$ is an idempotent cone with apex $m_\gamma$ (as in the proof of Lemma \ref{lemidinvcone}).  Now observe that
	$$\gamma \cdot\delta=\gamma \ast(\delta(c_{\gamma}))^\circ =\gamma\ast (\gamma(m_\gamma))^{-1}\gamma(m_\gamma)(\delta(c_{\gamma}))^\circ=\nu\ast \gamma(m_\gamma)(\delta(c_{\gamma}))^\circ.$$	
	Then since $\gamma(m_\gamma)$ is an isomorphism and $(\delta(c_{\gamma}))^\circ$ is an epimorphism, the morphism $\gamma(m_\gamma)(\delta(c_{\gamma}))^\circ$ is an epimorphism and has a unique normal factorisation of the form $qu$, where $q$ is a retraction and $u$ is an isomorphism. Let $\mu= \nu\ast q$. Then
	$$\gamma \cdot\delta=\nu\ast qu=(\nu\ast q)\ast u=\mu\ast u.$$
	By Lemma \ref{lemidinvcone}, the idempotent cone $\mu$ is inversive and so by Lemma \ref{leminvcone}, we  see that  $\gamma\cdot\delta$ is an inversive cone.
\end{proof}

Thus, given an inversive category $\mathcal{C}$,  the set $\widetilde{\mathcal{C}}$ of all its inversive cones forms a subsemigroup of the regular semigroup $\widehat{\mathcal{C}}$ of all its normal cones.

\begin{proposition}
	The subsemigroup $\widetilde{\mathcal{C}}$ is inverse.	
\end{proposition}
\begin{proof}
Given an inversive cone $\gamma\in\widetilde{\mathcal{C}}$ with apex $d$, Lemma~\ref{lemgc}(1) implies that every idempotent cone $\mu$ such that $\mu\mathrel{\mathscr{L}}\gamma$ in $\widehat{\mathcal{C}}$ satisfies $c_mu=d$. By (IC 5), there exists a unique idempotent inversive cone with this property. Hence every $\mathscr{L}$-class in $\widetilde{\mathcal{C}}$ contains a unique idempotent. Also, the $\mathscr{R}$-class of $\gamma$ in $\widetilde{\mathcal{C}}$ contains a unique idempotent, namely $\gamma\ast(\gamma(m_\gamma))^{-1}$, see Remark \ref{rmkuninvcone}. By Theorem \ref{thminv}(3), the semigroup $\widetilde{\mathcal{C}}$ is inverse.
\end{proof}

The following are immediate specialisations of Theorem \ref{thmls} and Corollary \ref{corls}.
\begin{thm}\label{thmvarPsi}
	Let $\mathcal{C}$ be an inversive category and $\widetilde{\mathcal{C}}$ the inverse semigroup of its inversive cones. Define a functor $\varPsi$ between the categories $\mathcal{C}$ and $\mathbb{L}(\widetilde{\mathcal{C}})$ as follows:
\begin{align*}
	v\varPsi(c)&:= \widetilde{\mathcal{C}}\mu&&\text{for }c\in v\mathcal{C},\\
	\varPsi(f)&:= \rho(\mu,\mu\ast f^\circ, \nu)&&\text{for }f\in\mathcal{C}(c,d),
\end{align*}
where $\mu,\nu\in E(\widetilde{\mathcal{C}})$ are such that $c_\mu=c$ and $c_{\nu}=d$. Then $\varPsi$ is an isomorphism of inversive categories.
\end{thm}
\begin{corollary}
	A category is inversive if and only if it is isomorphic to the category $\mathbb{L}(S)$ for some inverse semigroup $S$.	
\end{corollary}

\begin{rmk}
	Observe that we can prove the exact dual results for the category  $\mathbb{R}(S)$ of principal right ideals of an inverse semigroup $S$.	
\end{rmk}

Thus, given an inversive category $\mathcal{C}$, we have an associated inverse semigroup $\widetilde{\mathcal{C}}$. Now we proceed to show that this association is also functorial. To this end, we begin with the following lemma which easily follows from Lemma \ref{leminvcone}.
\begin{lemma}\label{lemphiinvcone}
	Suppose $\Phi$ is an inversive functor between two inversive categories $\mathcal{C}_1$ and $\mathcal{C}_2$, if $\gamma=\mu_{m_\gamma}\ast u$ is an inversive cone in $\mathcal{C}_1$, then
	$$\Phi(\gamma):=\mu_{v\Phi(m_\gamma)} \ast \Phi(u)$$
	is an inversive cone in the inversive category $\mathcal{C}_2$.
\end{lemma}
\begin{lemma}\label{lemphiinvcone1}
	If $\Phi$ is an inversive functor between two inversive categories $\mathcal{C}_1$ and $\mathcal{C}_2$, then the mapping $\phi:\widetilde{\mathcal{C}_1}\to \widetilde{\mathcal{C}_2}$ defined as
	$$\phi\colon\gamma\mapsto\Phi(\gamma)$$
	is a semigroup homomorphism.
\end{lemma}
\begin{proof}
	By Remark \ref{rmkuninvcone} and Lemma \ref{lemphiinvcone}, the mapping $\phi$ is well defined. Also, for $\gamma_1,\gamma_2\in \widetilde{\mathcal{C}_1}$
	\begin{align*}
	(\gamma_1\cdot\gamma_2)\phi
	&=\Phi(\gamma_1\cdot\gamma_2)\\
	&=\Phi(\gamma_1 \ast(\gamma_2(c_{\gamma_1}))^\circ)&&\text{using \eqref{eqnbin}}\\
	&=\Phi (\mu_{m_{\gamma_1}}\ast u_1\:(\mu_{m_{\gamma_2}}\ast u_2(c_{\gamma_1}))^\circ)&&\text{writing } \gamma_1=\mu_{m_{\gamma_1}}\ast u_1 \\
	& &&\text{and }\gamma_2=\mu_{m_{\gamma_2}}\ast u_2\text{ for}\\
    & &&\text{some isomorphisms }u_1,u_2\\
	&=\Phi (\mu_{m_{\gamma_1}}\ast u_1\:(\mu_{m_{\gamma_2}}(c_{\gamma_1})\: u_2)^\circ)&&\text{using \eqref{eqnbin0}}\\
	&=\Phi (\mu_{m_{\gamma_1}}\ast u_1\:q\: u)&&\text{where } qu \text{ is the}\\
	& &&\text{normal factorisation of }\\
	& && (\mu_{m_{\gamma_2}}(c_{\gamma_1})\: u_2)^\circ \\
	&=\Phi (\mu_{m_{\gamma_1}}\ast q'u')&&\text{where } q'u' \text{ is the}\\
	& &&\text{normal factorisation of }\\
	& &&\text{the epimorphism } u_1q u \\
	&=\Phi (\mu_{c'}\ast u')&&\text{letting } \mu_{c'}:=\mu_{m_{\gamma_1}}\ast q'\\
	&=\mu_{v\Phi(c')} \ast \Phi(u')&&\text{using Lemma \ref{lemphiinvcone}}\\
	&=\mu_{v\Phi(m_{\gamma_1})}\ast q(v\Phi(m_{\gamma_1}),v\Phi(c'))\:\Phi(u')&&\text{using Lemma \ref{lemidinvcone}}\\
	&=\mu_{v\Phi(m_{\gamma_1})}\ast \Phi(q')\:\Phi(u')&&\text{using (IC 2) in }\mathcal{C}_2\\
	&=\mu_{v\Phi(m_{\gamma_1})}\ast \Phi(u_1)\:\Phi(q)\Phi(u)&&\text{since inversive functors}\\
	& && \text{preserve normal factorisations}\\
	&=\mu_{v\Phi(m_{\gamma_1})}\ast \Phi(u_1)\:\Phi((\mu_{m_{\gamma_2}}(c_{\gamma_1})\: u_2)^\circ)&&\text{---\texttt{"}---}\\
	&=\mu_{v\Phi(m_{\gamma_1})}\ast \Phi(u_1)\:(\Phi(\mu_{m_{\gamma_2}}(c_{\gamma_1}))\: \Phi(u_2))^\circ&&\text{---\texttt{"}---}\\
	&=\mu_{v\Phi(m_{\gamma_1})} \ast \Phi(u_1) \:  (\mu_{v\Phi(m_{\gamma_2})}(v\Phi(c_{\gamma_1})) \:  \Phi(u_2))^\circ&&\text{using Lemma \ref{lemphiinvcone}}\\
	&=\mu_{v\Phi(m_{\gamma_1})} \ast \Phi(u_1) \:   (\mu_{v\Phi(m_{\gamma_2})} \: \ast \Phi(u_2)(v\Phi(c_{\gamma_1})))^\circ&&\text{using \eqref{eqnbin0}}\\
	&=(\mu_{v\Phi(m_{\gamma_1})} \ast \Phi(u_1))\cdot(\mu_{v\Phi(m_{\gamma_2})} \ast \Phi(u_2))&&\text{using \eqref{eqnbin}}\\
	&=\Phi(\gamma_1)\cdot\Phi(\gamma_2)&&\text{using Lemma \ref{lemphiinvcone}}\\
	&=\gamma_1\phi\cdot\gamma_2\phi.&&
	\end{align*}
\end{proof}

Thus we have the following proposition which can be easily verified.
\begin{proposition}\label{proS}
	Given an inversive category $\mathcal{C}$ and an inversive functor $\Phi$ between two inversive categories $\mathcal{C}_1$ and $\mathcal{C}_2$, the following assignment
	$$\mathcal{C}\mapsto \widetilde{\mathcal{C}}; \quad \Phi \mapsto \phi$$
	constitutes a functor, say $\mathtt{S}$, from the category $\mathbf{IC}$ of inversive categories to the category $\mathbf{IS}$ of inverse semigroups.
\end{proposition}

\subsection{A new structure theorem for inverse semigroups}
Having characterised the principal left (right) ideals of an inverse semigroup $S$ as an inversive category, we now proceed to establish a new structure theorem for inverse semigroups.
First recall the following folklore property of inverse semigroups.

\begin{proposition} \label{proinvrr}
	Let $S$ be an inverse semigroup. For every pair $p,q$ of distinct elements in $S$ there exists an idempotent $e\in S$ such that $pe\ne qe$ (respectively $ep\ne eq$). In particular, $S$ is right reductive.
\end{proposition}
\begin{proof}
	Suppose $S$ is inverse and $p,q\in S$ are such that $pe=qe$ for
	every idempotent $e\in S$. Taking the idempotent $p^{-1}p$ for $e$,
	we conclude that $p=pp^{-1}p=qp^{-1}p$ and, similarly, $q=pq^{-1}q$.
	So $$p=qp^{-1}p=pq^{-1}qp^{-1}p=pp^{-1}pq^{-1}q=pq^{-1}q=q.$$	
	Hence the proposition.
\end{proof}
\begin{thm}\label{thmrhoiso}
	If $S$ is an inverse semigroup, the map $\tilde{\rho}\colon S \to \widetilde{\mathbb{L}(S)}$ given by $a\mapsto\rho^a$ is a semigroup isomorphism.
\end{thm}
\begin{proof}
	Using the language of Proposition \ref{prosr}, Proposition \ref{proinvrr} implies that when $S$ is an inverse semigroup, the right regular representation $\rho\colon S\to S_\rho$ is injective. Hence by Proposition \ref{prosr}, the map $\tilde{\rho}\colon S \to \widetilde{\mathbb{L}(S)}$ given by $a\mapsto\rho^a$ is an injective homomorphism. Also, if $\gamma\in \widetilde{\mathbb{L}(S)}$ is any inversive cone, then by Lemma \ref{leminvcone}, $\gamma=\mu\ast u$ for some idempotent inversive cone $\mu$ and an isomorphism $u$ in $\mathbb{L}(S)$. Clearly by Lemma \ref{lempic}, the only idempotent inversive cones in $\mathbb{L}(S)$ are those of the form $\rho^e$ for some $e\in E(S)$. Also, since any isomorphism in $\mathbb{L}(S)$ is of the form $\rho(e,u,f)$ for $e\mathrel{\mathscr{R}}u\mathrel{\mathscr{L}}f$, we see that
	$$\gamma= \mu\ast u= \rho^e\ast \rho(e,u,f)=\rho^{eu}=\rho^u.$$
	So $\tilde{\rho}$ is surjective and hence the theorem.
\end{proof}

Now, by Proposition \ref{proC}, we have a functor $\mathtt{C}$ from the category $\mathbf{IS}$ of inverse semigroups to the category $\mathbf{IC}$ of inversive categories and by Proposition \ref{proS}, we have a functor $\mathtt{S}$ from the category $\mathbf{IC}$ of inversive categories to the category $\mathbf{IS}$ of inverse semigroups. We proceed to prove that the functors $\mathtt{C}$ and $\mathtt{S}$ constitute an adjunction between the categories $\mathbf{IS}$ and $\mathbf{IC}$, leading to a category equivalence.

Observe that for a given inverse semigroup $S$,
$$\mathtt{C}\mathtt{S} (S)=\mathtt{S}(\mathbb{L}(S))=\widetilde{\mathbb{L}(S)}.$$
So if we define the map	$\psi(S)\colon S\to \mathtt{C}\mathtt{S} (S)$ as $a\mapsto\rho^a$, then by Theorem \ref{thmrhoiso}, we see that $\psi(S)$ is a semigroup isomorphism. So, any element of $\mathtt{C}\mathtt{S} (S)$ can be denoted by $\rho^u$ for some $u\in S$. Also for $e\mathrel{\mathscr{R}}u\mathrel{\mathscr{L}}f$, we see that
$$\rho^u=\rho^{eu}= \rho^e\ast \rho(e,u,f).$$
Further, given a homomorphism $\phi$ between two inverse semigroups $S_1$ and $S_2$, using Proposition \ref{proC}, we see that  $\mathtt{C}(\phi) = \Phi$ as defined in  (\ref{eqnPhi}). So, we have a functor $\mathtt{C}(\phi) \colon \mathbb{L}(S_1)\to \mathbb{L}(S_2)$. Thus, $\mathtt{S}(\mathtt{C}(\phi))$ is the semigroup homomorphism as defined in Lemma \ref{lemphiinvcone} between the semigroups $\widetilde{\mathbb{L}(S_1)}$ and $\widetilde{\mathbb{L}(S_2)}$  induced by the functor $\mathtt{C}(\phi)$. Hence for an element $\rho^{u}= \rho^{e}\ast \rho(e,u,f) =\mu_{S_1e} \ast \rho(e,u,f)\in \widetilde{\mathbb{L}(S_1)}=\mathtt{C}\mathtt{S} (S_1)$,
\begin{align*}
(\rho^{u})\mathtt{C}\mathtt{S}(\phi) &= (\mu_{S_1e}\ast \rho(e,u,f))\mathtt{C}\mathtt{S}(\phi)\\
&= \mu_{v\mathtt{C}(\phi)(S_1e)}\ast \mathtt{C}(\phi)(\rho(e,u,f))\\
&= \mu_{v\Phi(S_1e)}\ast \Phi(\rho(e,u,f))\\
&=\mu_{S_2e\phi} \ast\rho(e\phi,u\phi,f\phi)\\
&=\rho^{e\phi}\ast\rho(e\phi,u\phi,f\phi)\\
&=\rho^{e\phi u\phi}=\rho^{e u\phi}= \rho^{u\phi}.
\end{align*}

\begin{proposition}\label{proeq1}
	Given an inverse semigroup $S$, the map $\psi\colon S\mapsto\psi(S)$ is a natural isomorphism from the identity functor $1_{\mathbf{IS}}$ to the functor $\mathtt{C}\mathtt{S}$.
\end{proposition}
\begin{proof}
	Since $\psi(S)$ is a semigroup isomorphism, it suffices to show that $\psi$ is a natural transformation. That is, for a homomorphism $\phi \colon S_1 \to S_2$, we need to show that the following diagram commutes.
	\begin{equation*}\label{}
	\xymatrixcolsep{3pc}\xymatrixrowsep{4pc}\xymatrix
	{
		S_1 \ar[d]_{\phi}\ar[rr]^{\psi(S_1)}   	&& \mathtt{C}\mathtt{S} (S_1) \ar[d]^{\mathtt{C}\mathtt{S} (\phi)} \\
		S_2 \ar[rr]^{\psi(S_2)}   	&& \mathtt{C}\mathtt{S} (S_2)
	}
	\end{equation*}
	For $a\in S_1$, we have
	$$a\phi\psi(S_2)=\rho^{a\phi}.$$
	Also, from the discussion above, we have,
	\begin{align*}\label{}
	a\psi(S_1)\mathtt{C}\mathtt{S} (\phi) = (\rho^a)\mathtt{C}\mathtt{S} (\phi)
	=\rho^{a\phi}.
	\end{align*}
	So, the above diagram commutes and hence $\psi$ is a natural isomorphism.	
\end{proof}

Finally, we need to show that the identity functor $1_{\mathbf{IC}}$ is naturally isomorphic to the functor $\mathtt{S}\mathtt{C}$. Observe that for a given inversive category $\mathcal{C}$,
$$\mathtt{S}\mathtt{C}(\mathcal{C})= \mathtt{C}(\widetilde{\mathcal{C}})=\mathbb{L}(\widetilde{\mathcal{C}}).$$
So if we define a functor $\varPsi(\mathcal{C})\colon \mathcal{C}\to \mathbb{L}(\widetilde{\mathcal{C}})$ as in Theorem \ref{thmvarPsi}, recall that $\varPsi(\mathcal{C})$ is an isomorphism of inversive categories. Given an inversive  functor $\Phi$ between two inversive categories $\mathcal{C}_1$ and $\mathcal{C}_2$, then $\mathtt{S}(\Phi)$ as defined in Lemma \ref{lemphiinvcone1} is a semigroup homomorphism from $\widetilde{\mathcal{C}_1}$ to $\widetilde{\mathcal{C}_2}$. Observe that an arbitrary object of the category $\mathtt{S}\mathtt{C} (\mathcal{C}_1)=\mathbb{L}(\widetilde{\mathcal{C}_1})$ may be denoted by $\widetilde{\mathcal{C}_1}\vartheta$ where $\vartheta$ is an idempotent inversive cone in $\mathbb{L}(\widetilde{\mathcal{C}_1})$. Hence the functor $\mathtt{S}\mathtt{C} (\Phi)$ will map an object $\widetilde{\mathcal{C}_1}\vartheta\mapsto\widetilde{\mathcal{C}_2}\Phi(\vartheta)$ where $\Phi(\vartheta)$ is defined as in Lemma \ref{lemphiinvcone}.

If $\rho(\vartheta,\xi,\upsilon)$ is an arbitrary morphism of the category $\mathbb{L}(\widetilde{\mathcal{C}_1})$ where $\vartheta,\upsilon$ are idempotent inversive cones in $\mathbb{L}(\widetilde{\mathcal{C}_1})$ and $\xi\in\vartheta \widetilde{\mathcal{C}_1} \upsilon$, the functor $\mathtt{S}\mathtt{C} (\Phi)$ map is as follows:
$$\rho(\vartheta,\xi,\upsilon)\mapsto\rho(\Phi(\vartheta),\Phi(\xi),\Phi(\upsilon)),$$
where $\Phi(\vartheta)$ etc. are defined as in Lemma \ref{lemphiinvcone}.

\begin{proposition}\label{proeq2}
	Given an inversive category $\mathcal{C}$, the map $\varPsi\colon \mathcal{C}\mapsto\varPsi(\mathcal{C})$ is a natural isomorphism from the identity functor $1_{\mathbf{IC}}$ to the functor $\mathtt{S}\mathtt{C}$.
\end{proposition}
\begin{proof}
	Since $\varPsi(\mathcal{C})$ is an isomorphism, it suffices to show that $\varPsi$ is a natural transformation. That is, for a functor $\Phi \colon \mathcal{C}_1 \to \mathcal{C}_2$, we need to show that the following diagram commutes.
	\begin{equation*}\label{}
	\xymatrixcolsep{3pc}\xymatrixrowsep{4pc}\xymatrix
	{
		\mathcal{C}_1 \ar[d]_{\Phi}\ar[rr]^{\varPsi(\mathtt{C}_1)}   	&& \mathtt{S}\mathtt{C} (\mathcal{C}_1) \ar[d]^{\mathtt{S} \mathtt{C} (\Phi)} \\
		\mathcal{C}_2 \ar[rr]^{\varPsi(\mathtt{C}_2)}   	&& \mathtt{S}\mathtt{C} (\mathcal{C}_2)
	}
	\end{equation*}
	
	For an object $c\in v\mathcal{C}_1$, we have
	\begin{align*}
	\Phi\varPsi(\mathcal{C}_2)(c)&=\varPsi(\mathcal{C}_2)(\Phi(c))\\
	&=\widetilde{\mathcal{C}_2}\mu \quad \text{(where $\mu$ is an idempotent inversive cone such that }c_\mu=\Phi(c).)
	\end{align*}
	Also, from the discusion above, we have,
	\begin{align*}\label{}
	\varPsi(\mathcal{C}_1)\mathtt{S}\mathtt{C}(\Phi) (c) &= \mathtt{S}\mathtt{C}(\Phi)(\widetilde{\mathcal{C}_1}\vartheta)\quad \text{(where } c_\vartheta=c)\\
	&=\widetilde{\mathcal{C}_2}\Phi(\vartheta) \quad (\text{where $\Phi(\vartheta)$ is an inversive cone as in Lemma \ref{lemphiinvcone}.})
	\end{align*}
	By the defnition of $\Phi(\vartheta)$ as in Lemma \ref{lemphiinvcone}, it is clear that if $\vartheta$ is an idempotent inversive cone in $\mathcal{C}_1$ with apex $c$, then $\Phi(\vartheta)$ is an idempotent inversive cone $\mathcal{C}_2$ with apex $\Phi(c)$. Moreover since there is a unique idempotent inversive cone with a given vertex, we conclude that $\mu=\Phi(\vartheta)$. Hence
	$$\Phi\varPsi(\mathcal{C}_2)(c)= \varPsi(\mathcal{C}_1)\mathtt{S}\mathtt{C}(\Phi) (c)$$ and so the diagram commutes for every object $c\in v\mathcal{C}_1$.
	
	Further for a morphism $f\colon c\to d$ in the category $\mathcal{C}_1$, we have
	\begin{align*}
	\Phi\varPsi(\mathcal{C}_2)(f)=\varPsi(\mathcal{C}_2)(\Phi(f))=\rho(\mu,\mu\ast(\Phi(f))^\circ,\nu)
	\end{align*}
	where $\mu$ and $\nu$ are idempotent inversive cones such that $c_\mu=\Phi(c)$ and $c_\nu=\Phi(d)$.
	Also,
	\begin{align*}
	\varPsi(\mathcal{C}_1)\mathtt{S}\mathtt{C}(\Phi) (f)&=\mathtt{S}\mathtt{C}(\Phi) (\rho(\vartheta,\vartheta\ast f^\circ,\upsilon)) &&\text{where }\vartheta,\upsilon\text{ are idempotent inversive cones}\\
	& &&\text{in }\mathbb{L}(\widetilde{\mathcal{C}_1})\text{ such that }c_\vartheta=c\text{ and }c_\upsilon=d\\
	&=\rho(\Phi(\vartheta),\Phi(\vartheta\ast f^\circ),\Phi(\upsilon)) &&\text{where }\Phi(\vartheta) \text{ etc. are as in Lemma \ref{lemphiinvcone}.}
	\end{align*}
	As argued in the case of objects, we can easily see that $\mu=\Phi(\vartheta)$ and $\nu=\Phi(\upsilon)$. Further, since $\Phi$ is an inversive functor, as argued in the proof of Lemma \ref{lemphiinvcone1}, we can verify that $\mu\ast(\Phi(f))^\circ=\Phi(\vartheta\ast f^\circ)$. That is, for a morphism $f\colon c\to d$ in the category $\mathcal{C}_1$,
	$$\Phi\varPsi(\mathcal{C}_2)(f)= \varPsi(\mathcal{C}_1)\mathtt{S}\mathtt{C}(\Phi) (f)$$ and  the diagram commutes for every morphism $f$ in $\mathcal{C}_1$. So, the diagram is commutative and hence $\varPsi$ is a natural isomorphism.	
\end{proof}

Combining Propositions \ref{proeq1} and \ref{proeq2}, we arrive at the main result of this section:
\begin{thm}\label{thminvcat}
	The category $\mathbf{IS}$ of inverse semigroups is equivalent to the category $\mathbf{IC}$ of inversive categories.	
\end{thm}

\subsection{Inductive groupoids and inversive categories}

Any element $x$ in an inverse semigroup $S$ can be seen as a \emph{morphism} from the idempotent $e:=xx^{-1}$ to the idempotent $f:=x^{-1}x$.
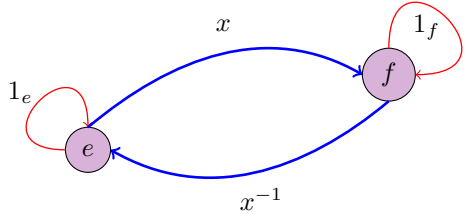
\begin{center}
	\begin{tikzpicture}[ node distance=0 cm,outer sep = 0pt]
	\node[draw, circle, minimum size=.6cm, fill=violet!30, node distance=1.75cm] (c1) at (-3,0) {$e$};
	\node[draw, circle, minimum size=.6cm, fill=violet!30, node distance=1.75cm] (c2) at (1,1) {$f$};
	\node[minimum height=.5cm, minimum width=.6cm, fill=white!20,anchor=south west] at (-4.2,.5) {$1_e$};
	\draw[->][line width=.5pt,red,in=90, out=180,looseness=8] (c1.west) to (c1.north);
	\node[minimum height=.5cm, minimum width=.6cm, fill=white!20,anchor=south west] at (1.2,1.35) {$1_f$};
	\draw[->][line width=.5pt,red,in=0, out=90,looseness=8] (c2.north) to (c2.east);
	\node[minimum height=.5cm, minimum width=.6cm, fill=white!20,anchor=south west] at (-1.5,1.4) {$x$};
	\draw[->][line width=1pt,blue,in=150, out=40] (c1.north) to  (c2.west);
	\node[minimum height=.5cm, minimum width=.6cm, fill=white!20,anchor=north west] at (-1.1,-.4) {$x^{-1}$};
	\draw[->][line width=1pt,blue,in=330, out=220] (c2.south) to  (c1.east);
	\end{tikzpicture}
\end{center}
In this manner, one can naturally associate a groupoid $\mathcal{G}(S)$ with a given inverse semigroup $S$ such that the set $v\mathcal{G}(S)$ of objects coincides with the set of idempotents of the semigroup $S$. Abstracting the characteristic properties of the groupoid $\mathcal{G}(S)$ leads to the following definition.

\begin{dfn}\label{ig}
	Let $\mathcal{G}$ be a groupoid and denote by $\mathbf{d}\colon\mathcal{G}\to v\mathcal{G}$ and $\mathbf{r}\colon\mathcal{G}\to v\mathcal{G}$ its domain and codomain maps, respectively. Let $\leq$ be a partial order on $\mathcal{G}$. Then $(\mathcal{G},\leq)$ is called an \emph{inductive groupoid} if $(v\mathcal{G},\leq)$ is a semilattice and for all $e,f \in v\mathcal{G}$ and all $x,y,u,v\in\mathcal{G}$, the following hold.
	\begin{enumerate}
		\item [(OG1)] If $u\leq x$, $v\leq y$ and $\mathbf{r}(u)=\mathbf{d}(v)$, $\mathbf{r}(x)=\mathbf{d}(y)$, then $uv \leq xy$.
		\item [(OG2)] If $x\leq y$, then $x^{-1}\leq y^{-1}$.
		\item [(OG3)] If $1_e\leq 1_{\mathbf{d}(x)}$, then there exists a unique morphism $e{\downharpoonleft}x\in\mathcal{G}$ (called the \emph{restriction} of $x$ to $e$) such that $e{\downharpoonleft} x\leq x$ and $\mathbf{d}(e{\downharpoonleft} x) = e$.
		\item [(OG3$^*$)] If $1_f\leq 1_{\mathbf{r}(x)}$, then there exists a unique morphism $x{\downharpoonright}f\in\mathcal{G} $ (called the \emph{corestriction} of $x$ to $f$) such that $x{\downharpoonright} f \leq x$ and $\mathbf{r}(x{\downharpoonright} f) = f$.
	\end{enumerate}
\end{dfn}

The inductive groupoids with inductive functors as morphisms form the locally small category  $\mathbf{IG}$ of inductive groupoids, see \cite[Chapter 4]{lawson} for details. The above discussed association between semigroups and groupoids can be extended to a category isomorphism (not just a category equivalence) as follows.
\begin{thm}[\mdseries Ehresmann--Schein--Nambooripad Theorem, see {\cite[Theorem 4.1.8]{lawson}}]\label{thmesn}
	The category $\mathbf{IG}$ of inductive groupoids is isomorphic to the category $\mathbf{IS}$ of inverse semigroups.
\end{thm}
Using Theorem \ref{thmesn} and Theorem \ref{thminvcat}, by transitivity, we see that the category $\mathbf{IG}$ of inductive groupoids is equivalent to the category $\mathbf{IC}$ of inversive categories. Now, we proceed to describe a direct category equivalence between the categories $\mathbf{IG}$ and $\mathbf{IC}$, without any semigroup assumptions. This may be seen as a very much simplified version of the results in \cite{indcxn1,indcxn2}.

First, given an inversive category $\mathcal{C}$ with the semilattice order $\le$, we proceed to identify the inductive groupoid associated with $\mathcal{C}$. To this end, let $\mathcal{G}_\mathcal{C}$ be the subcategory of the category $\mathcal{C}$ consisting of all isomorphisms in $\mathcal{C}$. Clearly $\mathcal{G}_\mathcal{C}$ is a groupoid. Given any two morphisms $f$ and $g$ in $\mathcal{G}_\mathcal{C}$ with domains $c$ and $d$ respectively, define a relation $\le_\mathcal{C}$ as follows:
$$f\le_\mathcal{C}g \iff c\le d \text{ and } f=(\iota(c,d)g)^\circ$$
where $\iota(c,d)$ is the inclusion from $c$ to $d$ and $(\iota(c,d)g)^\circ$ is the epimorphic component of the monomorphism $\iota(c,d)g$ in the inversive category $\mathcal{C}$. It can be easily seen that $\le_\mathcal{C}$ is a partial order on $\mathcal{G}_\mathcal{C}$.

Since the order $\le_\mathcal{C}$ reduces to the semilattice order $\le$ on the identities of $\mathcal{G}_\mathcal{C}$, we observe that $(v\mathcal{G}_\mathcal{C},\le_\mathcal{C})$ forms a semilattice. Further, given a morphism $g$ in $\mathcal{G}_\mathcal{C}$ with domain $d$ and if $1_c\le_\mathcal{C}1_d$, then by letting
$$c{\downharpoonleft}g:= (\iota(c,d)g)^\circ$$
as the restriction of the morphism $g$ in the groupoid $\mathcal{G}_\mathcal{C}$ to the object $c$, we can easily verify the following.

\begin{proposition}
	$(\mathcal{G}_\mathcal{C},\le_\mathcal{C})$ is an inductive groupoid.	
\end{proposition}

Also, given an inversive functor between two inversive categories, its restriction will give an inductive functor between the two corresponding inductive groupoids. Thus the above correspondence is functorial between the categories $\mathbf{IC}$ and $\mathbf{IG}$.

Conversely, given an inductive groupoid $(\mathcal{G},\leq)$, we proceed to `build' the inversive category $\mathcal{C}_\mathcal{G}$ associated with it. To this end, we will follow the scheme used in \cite[Section 4]{indcxn2}. We consider three intermediary categories: $\mathcal{P}_\mathcal{G}$, $\mathcal{G}$ and $\mathcal{Q}_\mathcal{G}$  responsible for inclusions, isomorphisms and retractions, respectively. Then we introduce a special partial binary operation to build the required category $\mathcal{C}_\mathcal{G}$ from these categories. A somewhat similar construction can be seen in \cite[Theorem 3.4]{rajancat}.

Given an inductive groupoid $(\mathcal{G},\leq)$, the categories $\mathcal{P}_\mathcal{G}$ and $\mathcal{Q}_\mathcal{G}$ are such that
$$v\mathcal{P}_\mathcal{G}= v\mathcal{Q}_\mathcal{G} :=v\mathcal{G}.$$
Also, for $e,f\in v\mathcal{G}$ such that $e\leq f$, a morphism in $\mathcal{P}_\mathcal{G}$ is defined as the unique morphism $\iota(e,f)$ from $e$ to $f$ and a morphism in $\mathcal{Q}_\mathcal{G}$ is defined as the unique morphism $q(f,e)$ from $f$ to $e$. The category $\mathcal{P}_\mathcal{G}$ is a strict preorder and its morphisms are called inclusions.

In the case of general normal categories in \cite{indcxn1,indcxn2}, a further auxiliary category was needed. Here we do not need it since inversive categories have unique factorisations and this allows us to use $\mathcal{G}$ itself as a building block. So, the set of objects of the required category is $v\mathcal{C}_\mathcal{G} :=v\mathcal{G}$ and morphisms in $\mathcal{C}_\mathcal{G}$ are defined by:
\begin{equation*}
\mathcal{C}_\mathcal{G}:=\{(q,\alpha,j)\in \mathcal{Q}_\mathcal{G} \times \mathcal{G} \times \mathcal{P}_\mathcal{G} : \mathbf{r}(q) =\mathbf{d}(\alpha) \text{ and }  \mathbf{r}(\alpha) =\mathbf{d}(j)\}.
\end{equation*}

As in \cite[Section 4]{indcxn2}, we  denote an arbitrary morphism $(q,\alpha,j)$ in $\mathcal{C}_\mathcal{G}$ by just $[e,\alpha,f\rangle$ where $e=\mathbf{d}(q)$ and $f=\mathbf{r}(j)$. Given two such morphisms $[e,\alpha,f\rangle, [f,\beta,g \rangle \in \mathcal{C}_\mathcal{G}$, we  compose them as follows. For $h=\mathbf{r}(\alpha)\wedge\mathbf{d}(\beta)$,
\begin{equation}\label{eqncomplg}
[e,\alpha,f\rangle  \:  [f,\beta,g \rangle  =  [  e,\alpha{\downharpoonright}h\cdot h{\downharpoonleft}\beta,g \rangle .
\end{equation}

Then $\mathcal{C}_\mathcal{G}$ forms a category such that $\mathcal{P}_\mathcal{G}$ is a strict preorder subcategory by identifying any morphism $j\in \mathcal{P}_\mathcal{G}$ with $[\mathbf{d}(j),1_{\mathbf{d}(j)},\mathbf{r}(j)\rangle \in \mathcal{C}_\mathcal{G}$. Similarly, $\mathcal{Q}_\mathcal{G}$ is a subcategory of $\mathcal{C}_\mathcal{G}$ by identifying any morphism $q\in \mathcal{Q}_\mathcal{G}$ with $[\mathbf{d}(q),1_{\mathbf{r}(q)},\mathbf{r}(q)\rangle \in \mathcal{C}_\mathcal{G}$.
It is easy to verify that $(\mathcal{C}_\mathcal{G},\mathcal{P}_\mathcal{G})$ satisfies (IC 1), (IC 2) and (IC 3).

Let $f= \iota(e_1,e_2)q(e_2,e_3)\cdots\iota(e_{2n-1},e_{2n}) q(e_{2n},e_{2n+1})$ be an arbitrary morphism in the core $\langle \mathcal{C}_\mathcal{G} \rangle$ so that
\begin{align*}
f&=[e_1,1_{e_1},e_2\rangle \: [e_2,1_{e_3},e_3\rangle\: [e_3,1_{e_3},e_4\rangle \cdots [e_{2n-1},1_{e_{2n-1}},e_{2n-1}\rangle\: [e_{2n},1_{e_{2n+1}},e_{2n+1}\rangle\\
&=[e_1,1_{e_1\wedge e_3},e_3\rangle\: [e_3,1_{e_3},e_4\rangle \cdots [e_{2n-1},1_{e_{2n-1}},e_{2n-1}\rangle\: [e_{2n},1_{e_{2n+1}},e_{2n+1}\rangle\\
&=[e_1,1_{e_1\wedge e_3},e_4\rangle \cdots [e_{2n-1},1_{e_{2n-1}},e_{2n-1}\rangle\: [e_{2n},1_{e_{2n+1}},e_{2n+1}\rangle\\
&=[e_1,1_{e},e_{2n+1}\rangle  \qquad\text{ where }e= \bigwedge\limits_{i=0}^n e_{2i+1}\\
&=[e_1,1_{e},e\rangle [e,1_{e},e_{2n+1}\rangle.
\end{align*}
Hence any morphism in $\langle \mathcal{C}_\mathcal{G} \rangle$ admits an inversive factorisation and so (IC 4) is satisfied.

Finally, given $\alpha\in \mathcal{G}$, if we define a map $r^\alpha:v\mathcal{C}_\mathcal{G} \to \mathcal{C}_\mathcal{G} $ as
$$r^\alpha\colon g\mapsto [  g, h{\downharpoonleft}\alpha, \mathbf{r}(\alpha) \rangle $$
where $h=g\wedge\mathbf{d}(\alpha)$, we can easily verify that $r^\alpha$ is an inversive cone with apex $\mathbf{r}(\alpha)$ and $m_{r^\alpha}=\mathbf{d}(\alpha)$. Hence for any object $e\in v\mathcal{C}_\mathcal{G}$, if we define  $r^{e}$ as
$$r^e\colon g\mapsto [  g, 1_h, e \rangle $$
where $h=g\wedge e$, then $r^{e}$ is the unique idempotent inversive cone with apex $e$. Thus (IC 5) is also satisfied. Hence we have the following proposition.
\begin{proposition}
	$(\mathcal{C}_\mathcal{G},\mathcal{P}_\mathcal{G})$ is an inversive category.
\end{proposition}

As in \cite[Section 4]{indcxn2}, we can easily show that the discussed correspondence between the categories $\mathbf{IG}$ and $\mathbf{IC}$ is also functorial. This leads to the following theorem.

\begin{thm}
	The category $\mathbf{IG}$ of inductive groupoids is equivalent to the category $\mathbf{IC}$ of inversive categories.	
\end{thm}

Once the equivalence of the categories $\mathbf{IG}$ and $\mathbf{IC}$ has been established, the equivalence between the categories of inverse semigroups and inductive groupoids becomes a consequence of Theorem~\ref{thminvcat}. Thus, we have recovered a weak version of the Ehresmann--Schein--Nambooripad Theorem from our consideration.  (Recall that the `full' Ehresmann--Schein--Nambooripad Theorem claims that the two latter categories are isomorphic rather than equivalent.)

\section{Completely 0-simple semigroups}\label{seccss}

In this section, we  discuss how the abstract construction described in Section \ref{seccxnlis} simplifies in the case of completely $0$-simple semigroups. This section may also be seen as a relatively straightforward generalisation of the discussion in \cite{css} wherein the cross-connection structure of completely simple semigroups were studied in great detail. So whenever an exact repetition of the argument suffices, without further comment we refer the reader to \cite{css} for the details of the results outlined here. The ensuing discussion having a relatively low entry threshold may also act as a gateway to cross-connection theory for beginners.

It is known \cite{rees} that a completely 0-simple semigroup is isomorphic to the Rees matrix semigroup $\mathscr{M}^\circ(G;I, L ;P)$ described as follows. We set $\mathscr{M}^\circ(G;I, L ;P) = (I \times G \times  L ) \cup\{0\}$ where $G$ is a group, $I$ and $ L $ are sets and $P = (p_{\ell i})$ is an $L\times I$ matrix with entries in $G^\circ:= G\cup \{0\}$ such that no row or column of $P$ consists entirely of zeros (then $P$ is called a \emph{regular sandwich matrix}). The multiplication in $\mathscr{M}^\circ(G;I, L ;P)$ is given by:
\begin{equation}
\begin{aligned}
&(i,a,\ell) \: (j,b, k ) =
\begin{cases}
(i,ap_{\ell j}b, k ) &\text{ if }p_{\ell j} \neq 0, \\
0 &\text{ if } p_{\ell j} = 0,
\end{cases}\\
&(i,a,\ell) \: 0 = 0 \: (i,a,\ell) = 0 \: 0 = 0.
\end{aligned}
\end{equation}

Now, we proceed to specialise the results in Section \ref{seccxnlis}. To that end, we fix a Rees matrix semigroup $\mathscr{M}^\circ(G;I, L ;P)$ and denote it by just $S$. Then the category $\mathbb{L}(S)$ of principal left ideals may be described as follows.
\begin{equation}
v\mathbb{L}(S) =  L ^\circ: = L  \cup \{0\}
\end{equation}
In the sequel, any arbitrary principal left ideal $S(i,a,\ell)$ shall be represented by just $\ell$ whenever there is no confusion. Then for $\ell_1,\ell_2\in v\mathbb{L}(S)$ such that $\ell_1,\ell_2\ne 0$ and for each $g\in G^\circ$, an arbitrary morphism from $\ell_1$ to $\ell_2$ is given by $\rho(\ell_1, g,\ell_2)$ such that for each $(i,a,\ell_1) \in S(i_1,a_1,\ell_1)$,
$$\rho(\ell_1, g,\ell_2)\colon (i,a,\ell_1) \mapsto (i,g,\ell_2) \in S(i_2,a_2,\ell_2).$$
Thus $\mathbb{L}(S)(\ell_1,\ell_2) = G^\circ$ if $\ell_1,\ell_2\ne 0$.

Observe that $ L ^\circ$ is a strict preorder and the only non-trivial inclusions arise from the relation $0\subseteq\ell$. So, for each $\ell\ne 0$ in the set $ L $, we have an inclusion $\rho(0,0,\ell)$ from $0$ to $\ell$. Further, for each inclusion $\rho(0,0,\ell)\in \mathbb{L}(S)(0,\ell)$, we have a \emph{unique} retraction $\rho(\ell,0,0)\in \mathbb{L}(S)(\ell,0)$ and thus every inclusion in $\mathbb{L}(S)$ splits uniquely. Finally, the only morphism in $\mathbb{L}(S)(0,0)$ may be denoted by $\rho(0,0,0)$. Thus the composition of the morphisms  in $\mathbb{L}(S)$ is as described in the earlier sections:
$$\rho(\ell_1,g,\ell_2)\:\rho(\ell_2,h,\ell_3)=\rho(\ell_1,gh,\ell_3)$$
where $\ell_1,\ell_2,\ell_3\in L ^\circ$ and $g,h\in G^\circ$.
In the sequel, whenever there is no ambiguity regarding the domain and codomain of the morphism, we  represent an arbitrary morphism $\rho(\ell_1,g,\ell_2)$ in $\mathbb{L}(S)$ by just $\rho_g$. It can be easily verified that $\mathbb{L}(S)$ as described above forms a category with subobjects.

Observe that an arbitrary morphism $\rho(\ell_1,g,\ell_2) \in \mathbb{L}(S)$ either is an isomorphism wherein $\ell_1,\ell_2,g \ne0$ or has a unique normal factorisation of the form
$$\rho(\ell_1,0,\ell_2)=\rho(\ell_1,0,0)\rho(0,0,\ell_2).$$
Hence every morphism in $\mathbb{L}(S)$ has a unique normal factorisation.

Now, we proceed to describe the normal cones in $\mathbb{L}(S)$. As usual, given sets $A$ and $B$, the set of all functions $f\colon B\to A$ is denoted $A^B$. It is handy to represent $f\in A^B$ as a $|B|$-tuple of elements of $A$ of the form ${(a_\alpha)}_{\alpha\in B}$, where $a_\alpha:=f(\alpha)$.

The set $T:=(G^\circ)^ L  \times  L $ forms a semigroup under the following multiplication. Given $\gamma = ({(g_\alpha)}_{\alpha\in L};\ell), \:\delta = ({(h_\alpha)}_{\alpha\in L};k) \in T$ such that ${(g_\alpha)}_{\alpha\in L}, {(h_\alpha)}_{\alpha\in L} \in (G^\circ)^ L $ and $\ell, k  \in  L $,
\begin{equation}\label{eqnsgls1}
\gamma \ast \delta  : = ({(g_\alpha h_{\ell})}_{\alpha\in L}; k )
\end{equation}
where $h_{\ell} \in G^\circ$ and $g_\alpha h_{\ell}$ is the product in $G^\circ$.

It is clear that normal cones in $\mathbb{L}(S)$ can be represented as unique elements in $T$: any normal cone $\gamma$ with apex $\ell\ne 0$ can be represented by the element $({(g_\alpha)}_{\alpha\in L};\ell)$ in $T$ such that for an arbitrary $ k \in \mathbb{L}(S)$,
$$\gamma ( k ) =\rho( k ,g_ k ,\ell) \text{ if } k \ne 0 \text{ and } \gamma (0) =\rho(0,0,\ell).$$

The set $U=\{0\}^ L \times  L $ is an ideal of $T$. It consists of elements of $T$ of the form $({(0_\alpha)}_{\alpha\in L};\ell)$ which are not normal cones since none of their components are isomorphisms. Taking the Rees quotient $R:=T/U$, we can verify that the semigroup $\widehat{\mathbb{L}(S)}$ of normal cones in $\mathbb{L}(S)$ is isomorphic to $R$ whereas the zero of $R$ corresponds to the unique normal cone $\gamma_0$ in $\mathbb{L}(S)$ with apex $0$, namely $\gamma_0 (\ell)= \rho(\ell,0,0).$

We mention in passing that the semigroup $T$ is nothing but the \emph{wreath product} $G^\circ \wr  L$ of the $0$-group $G^\circ$ and the right zero semigroup $L$. (Thus the semigroup $\widehat{\mathbb{L}(S)}$ is isomorphic to the Rees quotient $(G^\circ \wr  L ) / (\{0\} \wr  L )$.) This observation reflects, in a nutshell, the fact that Nambooripad's somewhat mysterious composition of normal cones \eqref{eqnbin} is actually a sort of wreath product multiplication extended to a general category setting.

In the sequel, by abuse of notation, the image of an element $\gamma\in T$ in the quotient semigroup $R$ will also be denoted by just $\gamma$, whenever there is no confusion. So, for an arbitrary element $x=(i,a,\ell)$ in $S$, the principal cone $\rho^x$ in $\widehat{\mathbb{L}(S)}$ may be denoted by $({(p_{\alpha i} a)}_{\alpha\in L};\ell) \in (G^\circ)^ L  \times  L $. This normal cone $\rho^x$ with apex $\ell$ is obtained by the right translation of the $i$-th column of $P$ with the non-zero group element $a$. Since the sandwich matrix $P$ is \emph{regular}, we have $\rho^x\in T$.

Now, we proceed to characterise the Green relation $\mathscr{R}$ in the semigroup $\widehat{\mathbb{L}(S)}$; this in turn will provide the description of the unambiguous dual $\mathbb{L}(S)^*$. Extending the discussion in \cite{css}, we can see that given an arbitrary element $\gamma = ({(g_\alpha)}_{\alpha\in L};\ell) \in \widehat{\mathbb{L}(S)}$, the principal right ideal  $\gamma\widehat{\mathbb{L}(S)}$ generated by $\gamma$ is determined by the $ |L| $-tuple $g_\gamma= {(g_\alpha)}_{\alpha\in L} \in (G^\circ)^ L $ and
$$ \gamma\widehat{\mathbb{L}(S)}= g_\gamma G^\circ \times  L .$$

Thus, the $\mathscr{R}$-classes in the semigroup $\widehat{\mathbb{L}(S)}$ are in one-one correspondence with the set $\{g_\gamma G^\circ: g_\gamma\in(G^\circ)^ L \}$. Observe that the element ${(0_\alpha)}_{\alpha\in L}G^\circ$ corresponds to the $\mathscr{R}$-class generated by the zero element $\gamma_0$.

Thus the set of objects (i.e.\ the set of $H$-functors $\{H(\gamma;-):\gamma\in \widehat{\mathbb{L}(S)}\}$) of the normal dual $\mathbb{L}(S)^*$ may be characterised as the set $\{g_\gamma G^\circ: g_\gamma\in(G^\circ)^ L \}$.  Furthermore, using the fact that the $H$-functors are representable, we can show that the set of morphisms in $\mathbb{L}(S)^*$ are in a `dual' correspondence with the set of morphisms in $\mathbb{L}(S)$. So the morphisms between any two non-zero objects in $\mathbb{L}(S)^*$ can be characterised as the set $(G^\circ)^\text{op}$ and in the sequel we  denote such an arbitrary morphism by $\sigma_g$ where $g\in G^\circ$. Similarly we can describe the morphisms with the zero object $H(\gamma_0;-)$ as $\sigma_0$.

Having described the categories $\mathbb{L}(S)$ and $\mathbb{L}(S)^*$, using the left-right duality, we can easily describe the categories $\mathbb{R}(S)$ and $\mathbb{R}(S)^*$ as follows.
\begin{equation}
v\mathbb{R}(S) = I^\circ :=I \cup \{0\}.
\end{equation}
Any arbitrary principal right ideal $(i,a,\ell)S$ is represented by just $i$ and the set of morphisms $\mathbb{R}(S)(i_1,i_2) = (G^\circ)^\text{op}$ if $i_1, i_2\ne 0$ wherein an arbitrary morphism is denoted by just $\lambda_g$ where $g\in G^\circ$. If $0\subseteq i$, we have the associated inclusion morphism $\lambda(0,0,i)$ and the retraction $\lambda(0,0,i)$; the unique morphism in $\mathbb{R}(S)(0,0)$ is $\lambda(0,0,0)$.

Then it can be easily shown that the semigroup of normal cones in $\mathbb{R}(S)$ is given by the Rees quotient $((G^\circ)^I \times I)^\text{op}/(\{0\}^I \times I)^\text{op}$. Then the object set of the unambiguous dual $\mathbb{R}(S)^*$ is characterised as $\{ G^\circ g_\gamma: g_\gamma\in(G^\circ)^ I \}$ and an arbitrary morphism between any two non-zero objects is denoted by $\tau_g$ where $ g\in G^\circ$. In the sequel, we  refer to the unambiguous categories arising from a completely $0$-simple semigroup as \emph{completely unambiguous} categories.

Having constructed the completely unambiguous categories $\mathbb{L}(S)$ and $\mathbb{R}(S)$ associated with a completely simple semigroup $S$ using the sets $L, I$ and the $0$-group $G^\circ$, now we proceed to characterise the cross-connection involved.

Observe that given the $L \times I$ matrix $P$ with entries from $G^\circ$, we can define functors $\Gamma \colon \mathbb{R}(S) \to \mathbb{L}(S)^*$ and $\Delta \colon \mathbb{L}(S) \to \mathbb{R}(S)^*$ such that
\begin{gather}\label{eqng}
v\Gamma_P:{i}  \mapsto p_i  G^\circ,  \quad    v\Gamma_P:0\mapsto \gamma_0 G^\circ, \quad \Gamma_P:\: \lambda_g \mapsto \sigma_g \quad \text{ and } \quad \Gamma_P:\: \lambda_0 \mapsto \sigma_0;\\
\label{eqnd}
v\Delta_P:{\ell}  \mapsto G^\circ p_{\ell},  \quad  v\Delta_P:0\mapsto  G^\circ \delta_0, \quad \Delta_P:\: \rho_g \mapsto \tau_g \quad \text{ and } \quad \Delta_P:\: \rho_0 \mapsto \tau_0
\end{gather}
where $p_i$ is the $i$-th column of the matrix $P$ and  ${p_\ell}$ is the $\ell$-th row of $P$. It can be readily seen that these functors constitute a cross-connection. In other words, the sandwich matrix $P$ of the semigroup $\mathscr{M}^\circ(G;I, L ;P)$ completely determines the cross-connection functors. The discussion in this section can be
summarised in the following theorem.

\begin{thm}
Every completely $0$-simple semigroup $S=\mathscr{M}^\circ(G;I, L ;P)$ determines a cross-connection $(\mathbb{L}(S),\mathbb{R}(S);\Gamma_P,\Delta_P)$ between the completely unambiguous categories $\mathbb{L}(S)$ and $ \mathbb{R}(S)$ as described above. Conversely, given two arbitrary sets $L, I$ and a $0$-group $G^\circ$, we can define two completely unambiguous categories $\mathcal{C}$ and $\mathcal{D}$ such that an arbitrary \emph{regular} $L \times I$ matrix $P$ with entries from $G^\circ$ will determine a unique cross-connection $(\mathcal{C},\mathcal{D};\Gamma_P,\Delta_P)$; the cross-connection semigroup so obtained is the completely $0$-simple semigroup $\mathscr{M}^\circ(G;I, L ;P)$.
\end{thm}

The above described correspondence can be extended to a category equivalence between the category of completely $0$-simple semigroups and the category of cross-connections of completely unambiguous categories.

\section{Conclusion and future work}
\label{secconcl}

We have applied Nambooripad's theory of cross-connections to locally inverse semigroups, obtaining a category equivalence between the category of cross-connected unambiguous categories and the category of locally inverse semigroups. Being specialised to inverse semigroups, our main result leads to a new structure theorem that, as we have shown, is equivalent to a weak version of the Ehresmann--Schein--Nambooripad Theorem; if specialised to completely 0-simple semigroups, the result turns out to constitute a category-theoretic formulation of the Rees Theorem. Thus, we see that the cross-connection approach reveals a common background for the two classical and seemingly unrelated structure theorems.

This observation suggests a promising direction for further investigations aimed to find cross-connections explanations for several known (and possibly yet unknown) structure results for the class of locally inverse semigroups and its important subclasses. In particular, we mean the covering theorems by McAlister~\cite{mcal} and Pastijn--Oliveira~\cite{pastoliv} that are formulated in terms of the Rees matrix construction.

Considerable attention has been paid in the literature to the structure of the bifree locally inverse semigroup on a set and several related constructs~\cite{auinger93,auinger94,auinger95, MR1484423,MR3669784,oliveira}. Our result implies that these constructs possess categorical counterparts, which, in our opinion, are worth being studied: this might shed new light on the structure of bifree locally inverse semigroups and their relatives and reveal categorical creatures that might be of independent interest.

A rather unexpected application of the Ehresmann--Schein--Nambooripad Theorem has been recently demonstrated by Malandro~\cite{malandro} who used it for an efficient enumeration of finite inverse semigroups. We wonder if our main result can be applied in a similar fashion for an enumeration of finite locally inverse semigroups.

\end{document}